\documentclass[11pt]{amsart}

\usepackage{amscd,amsfonts,amssymb}

\usepackage{graphicx}
\usepackage{color}
\usepackage{hyperref}
\usepackage{nameref}
\usepackage{tcolorbox}

\usepackage[inner=2cm,outer=1.5cm,bottom=2.5cm,top=2cm]{geometry}

\allowdisplaybreaks

\usepackage{xcolor}

\definecolor{darkblue}{rgb}{0,0,102}

\date{\today}

\newtheorem{theorem}{Theorem}[section]
\newtheorem{proposition}[theorem]{Proposition}

\theoremstyle{definition}
\newtheorem{defn}[theorem]{Definition}
\newtheorem{lemma}[theorem]{Lemma}
\newtheorem{cor}[theorem]{Corollary}
\theoremstyle{remark}
\newtheorem{remark}[theorem]{Remark}

\numberwithin{equation}{section}

\newcommand\R{\mathbb{R}}
\newcommand\C{\mathbb{C}}
\newcommand{\RN}{\mathbb{R}^N}

\renewcommand{\Re}{\operatorname{Re}}

\begin{document}

\title[Carleman inequalities for fractional operators]{Carleman type inequalities  for fractional relativistic operators}

\author[L. Roncal, D. Stan and L. Vega]{Luz Roncal, Diana Stan and Luis Vega}

\address[L. Roncal]{BCAM - Basque Center for Applied Mathematics \\
48009 Bilbao, Spain\\
and  Ikerbasque, Basque Foundation for Science, 48011 Bilbao, Spain\\
 and Department of Mathematics\\
 UPV/EHU\\
Apto. 644, 48080 Bilbao, Spain}
\email{lroncal@bcamath.org}

\address[D. Stan]{Department of Mathematics, Statistics and Computation\\
University of Cantabria\\
39005 Santander, Spain}
\email{diana.stan@unican.es}

\address[L. Vega]{Department of Mathematics\\
 UPV/EHU\\
Apto. 644, 48080 Bilbao, Spain\\
and BCAM - Basque Center for Applied Mathematics \\
48009 Bilbao, Spain}
\email{luis.vega@ehu.es}



\begin{abstract}
In this paper we  derive Carleman estimates for the fractional relativistic operator.
We consider changing-sign solutions to the  heat equation for such operators. We prove monotonicity inequalities and convexity of certain energy functionals to deduce Carleman estimates with linear exponential weight. Our approach is based on spectral methods and functional calculus.
\end{abstract}
\maketitle


\section{Introduction and main result}

In this paper we are interested in exponential decay estimates for  strong solutions of the evolution equation 
\begin{equation}
\label{eq:proRel}
\begin{cases}
u_t(t,x)+(-\Delta + m^2)^s u(t,x)=V(t,x) u(t,x), &\quad x\in \RN, \,\,\, t>0, \\
u(0,x)=u_0(x), &\quad x\in \RN,
\end{cases}
\end{equation}
where $s\in (0,\frac{1}{2})$   and $m> 0$.  Here, the solution will be taken as $u:[0,T] \times \RN \to \R$  and the potential $V: [0,T]\times \RN \to \mathbb{R}.$

We will name the operator $(-\Delta + m^2)^s $ the \textit{fractional relativistic Schr\"odinger operator with mass $m$} (or just fractional relativistic operator). For $s=1/2$, $(-\Delta + m^2)^{1/2} $ is
sometimes called the square root Klein--Gordon operator, see for instance \cite{L} and to be more precise, the terminology \textit{relativistic Schr\"odinger operator} concerns $
\sqrt{-\Delta + m^2}-m+V$.
The latter is motivated by the kinetic energy of a relativistic particle (that is, a particle travelling with speed close to the speed of light $c$) with mass $m$, and $V$
corresponds to the quantization of the potential energy. We refer the reader to \cite{CMS}, where in particular the motivation and justification for the nomenclature of this
operator is explained in the Introduction. The relativistic Schr\"odinger operator has been extensively studied (\cite[Section 7.11]{LL}, \cite{FLS}) as well as the evolution
Problem \eqref{eq:proRel} which involves such an operator, see e.g. \cite{BDQ}.  We will not give an exhaustive account of the references.
Related equations to \eqref{eq:proRel} have been also considered, namely, the boson star equation was studied in \cite{FJL}.

\subsection{Motivation and main results}One of our main motivations is the search of lower bounds  for solutions of \eqref{eq:proRel} for large $|x|$ very much in the spirit of what is known for solutions of the heat equation. More concretely, consider
the heat equation with a potential
\begin{equation}
\label{eq:H}
\begin{cases}
u_t(t,x)-\Delta u(t,x)=V(t,x) u(t,x), &\quad x\in \RN, \,\,\, t>0, \\
u(0,x)=u_0(x), &\quad x\in \RN.
\end{cases}
\end{equation}
It was proved in  \cite{EKPV-JEMS,EKPV-CMP} that if the potential $V$ is bounded and $\|e^{\alpha (0) |\cdot|^2}u_0 \|_{L^2(\RN)}+\| e^{\alpha(T)|\cdot|^2}u(T,\cdot)\|_{L^2(\RN)}<~\infty$ with $\alpha (0)=0$ and $\alpha(T) =1/ (4T)$,  then $u\equiv 0$. This is known as a unique continuation result. The proof is obtained by contradiction after getting first lower bounds that hold for all the solutions of \eqref{eq:H} and which are obtained after a three-step procedure:
\begin{enumerate}
\item First, it is necessary to establish a monotonicity argument that gives  the persistence of the Gaussian decay for positive times if the same is assumed at the initial time. In this first step it is proved that $\alpha (t)$ can decrease with time, as it happens for example with the fundamental solution of \eqref{eq:H}.
\item The second step involves convexity arguments. It is proved that if the solution has the same decay at two different times $t_1<t_2$ (i.e. $\alpha(t_1)=  \alpha(t_2)$), then for $t_1<t<t_2$  the solution has a better gaussian decay, i.e.  $\alpha(t)> \alpha(t_1)$ for $t_1<t<t_2$.
\item Finally, the last ingredient is to obtain Carleman estimates that together with a localization procedure allow to prove the desired lower bounds.
\end{enumerate}

This procedure has turned out to be rather general. On the one hand, it was proved in \cite{EKPV07} that they work for local evolution equations of higher degree like the generalized Korteweg-De Vries equation, which is a third order equation in the spatial variable. In that example the right decay is superlineal and it is given by $e^{-\rho|x|^{3/2}}$ for some $\rho>0$. On the other hand, the same procedure was also proved to be successful in the discrete case. Indeed, when one considers the discrete Laplacian \cite{F-BV}, the decay rate is slightly superlineal and is given by $e^{-\rho|x|\log (1+|x|)}$. 

In this article we explore up to what extent the  three ingredients mentioned above hold for solutions of~\eqref{eq:proRel}. Let us describe the structure of the paper.
In Section \ref{Sect:Defn} we present several definitions of the operator, the associated fractional Leibniz rule and positivity inequalities. Moreover, we construct a family of explicit eigenfunctions. In Section \ref{Sect:Heat} we consider the initial value problem  for the associated homogeneous heat equation 
\begin{equation}
\label{eq:ReHeat}
\begin{cases}
u_t(t,x)+(-\Delta + m^2)^s u(t,x)=0, &\quad x\in \RN, \,\,\, t>0, \\
u(0,x)=u_0(x), &\quad x\in \RN,
\end{cases}
\end{equation}
with $s\in (0,1)$ and   $m> 0$.
  A remarkable property (consequence of $m>0$) that we show is that the following weighted energy is finite and log-convex
  \begin{equation}\label{energy:lambda}
  \mathcal{H}(t):=  \int_{\RN} e^{\lambda \cdot x} u^2(t,x) dx < \infty  \quad \text{for all }t>0  \quad \text{ and } \quad |\lambda|\le 2m,\end{equation}
  provided the same spatial decay is assumed for the initial data.

Then,  Section \ref{Sect:HeatPotential} is devoted to the study of strong solutions to Problem \eqref{eq:proRel}.   These are solutions $u\in C([0,T]:\mathbb{H}^{2s}(\RN))$ such that $u_t \in L^1([0,T]: \mathbb{H}^{2s}(\RN))$ and their existence is justified for data $u_0 \in   \mathbb{H}^{2s}(\RN)$ and $V\in L^\infty([0,T]\times \RN)$.  Here, $\mathbb{H}^{2s}(\RN)$ is the Sobolev space introduced in Section \ref{Sect:Defn}.
We prove that the quantity \eqref{energy:lambda} is still finite in the case of Problem \eqref{eq:proRel} with  a bounded potential and we also show a result concerning backward unique continuation for the heat equation with potential.  

Finally, in Section \ref{Sect:Carleman1}, we show Carleman-type estimates with linear exponential weight for solutions to Problem~\eqref{eq:proRel}. This is stated in the theorem below, where the operator $H_m^{s}$ is defined  in Proposition \ref{prop:leibniz}   
as
$$H^s_m(u,u)(t,x):=-C_{N,s} m^{\frac{N+2s}{2}}\int_{\mathbb{R}^N} \frac{\big(u(t,x)-u(t,y)\big)^2} {|x-y|^{\frac{N+2s}{2}}} K_{\frac{N+2s}{2}}(m|x-y|)\,dy -m^{2s}u^2(t,x).
$$ 
Here, $K_{\nu}(z)$ denotes the Macdonald's function of order $\nu$, see Section \ref{Sect:Defn}. 

  \begin{theorem}
  \label{th11}
Let $N\ge 1$, $s\in (0,1/2)$, $m>0$,  and let $u_0 \in   \mathbb{H}^{2s}(\RN) \cap   L^2(e^{\lambda\cdot x}\, dx)$ for some $\lambda \in \RN$ with $|\lambda|\le m$. Assume that $F(t,x):=V(t,x)u(t,x) \in L^2(0,T:L^2(e^{\lambda\cdot x}\, dx))$. Let $u$ be a  strong   solution to the initial value problem \eqref{eq:proRel} and let  $\omega(t,x)=e^{At}e^{\lambda\cdot x}$.
Then, the following inequality holds
  \begin{align*}
&\big((-|\lambda|^2+m^2)^s-A\big)^2  \int_0^1 t(1-t) \int_{\RN} \omega(t,x) u^2(t,x) dx\, dt +\frac{1}{2}\int_0^1\int_{\RN} \omega(t,x) u^2(t,x) dx\,dt \\
&\quad+ \frac{1}{2}\int_0^1 t(1-t)   \Big\{2  \int_{\RN} \omega ( u_t)^2dx  - \int_{\RN} \omega H^{2s}_m(u,u) dx  + (A+m^{2s})\int_{\RN} H^s_m(u,u) \omega \, dx\Big\} dt  \\
&\quad\le \int_{\RN} e^{\lambda \cdot x} (u_0^2(x) + u^2(1,x)) dx
 +  C_1(N,s)  \int_0^1 \int_{\RN} \omega(t,x)\big( (\partial_t+(-\Delta+m^2)^s)(u(t,x)) \big)^2\,dx\,dt,
\end{align*}
for  $A+m^{2s}<0$ sufficiently small (that is, $|A|$ sufficiently large) satisfying
$$
\big((-|\lambda|^2+m^2)^s-A\big)^2 \ge C_2(N,s) m^{4s},
$$
where $C_1(N,s), C_2(N,s)$ are positive constants depending only on $N$ and $s$.
\end{theorem}
The estimate in Theorem \ref{th11} is a Carleman-type inequality as mentioned in (3) above, which is achieved by: (1) proving monotonicity estimates for the corresponding energy functionals for the solution $u: [0,T]\times  \RN  \to \R$, see Proposition \ref{prop:monot}, and (2) convexity arguments, see Proposition \ref{prop:lower}.  Only linear exponential weights are admissible, and only for $m>0$,  a fact that  is strongly related to the decay of the kernel associated to the fractional relativistic operator, and in big contrast with the polynomial decay of the kernel of $(-\Delta)^s$. 

Our techniques are based on functional and spectral calculus and they hold for exponents $0<s<1/2$. Observe that this restriction in $s$ is due to the fact that we need the operator $H_m^{2s}$ to be negative to keep the corresponding energy term positive, and we are able to ensure this only by using the definition given in Proposition \ref{prop:leibniz}, which is valid for $0<s<1/2$. We notice a major difference with respect to the classical diffusion process: the coefficient $\alpha(t)$ in the weight $e^{\lambda \cdot x}$ is always constant, thus the spatial decay does not change with time.  Log-convexity still works, thus the decays at time $0$ an $1$ control the decay at intermediate times. We also notice the important restriction $|\lambda|\le m$. As it is shown in Proposition \ref{prop:analy}, the weight $e^{\lambda \cdot x}$ is an eigenfunction of the operator $L_m^s$ under the restriction of $|\lambda|<m$ (the case $|\lambda|=m$ is trivial). There is a deeper obstruction behind, related with the analyticity of the Fourier multiplier in the case $|\lambda|>m$, see Remark \ref{rem:analyti}.

Unfortunately, the Carleman estimates stated in Theorem \ref{th11} are not sufficient to conclude any lower bounds. The problem comes from the non-local properties of the operator $(-\Delta+m^2)^s$. As we said in (3) above in order to use the Carleman estimates, a localization procedure is necessary. When this is done in the non-local setting, the only way of closing the argument is to assume that the fundamental solution in the constant coefficient case \eqref{eq:ReHeat} decays at least as a superlineal exponential, which is not the case. For example
the fundamental solution (see Section \ref{Fundamental:Solution}) for $s=1/2$ is known to have an explicit form  (see \cite{Ba,CMS}, also \cite[p. 185--186]{LL}):
$$
K_t^{1/2}(x)=(2\pi)^{N/2}\Big(\frac{2}{\pi}\Big)^{1/2}m^{\frac{N+1}{2}}t(|x|^2+t^2)^{-\frac{N+1}{4}}K_{\frac{N+1}{2}}(m\sqrt{|x|^2+t^2}),
$$
and $K_t^{1/2}(x) \sim e^{-m|x|}$ for large $|x|$, see \eqref{eq:asympInf}.

This rises the question of some other possible scenarios of non-local operators that exhibit super-lineal exponential decay.
There  are examples of distribution densities of L\'evy processes which show a ``weakly super-linear'' asymptotic behaviour. Let us explain this more precisely: let $(Z_t)_{t\ge0}$ be a real--valued L\'evy process with characteristic exponent $\psi$, i.e., $Ee^{iz Z_t}=e^{t\psi(z)}$, $t>0$. The function $\psi:\R\to \C$ admits the L\'evy--Khinchin representation
$$
\psi(z)=iaz-bz^2+\int_{\R}(e^{iuz}-1-izu 1_{\{|u|\le1\}})\,\mu(du), \quad a\in \R, \,\,\, b\ge 0
$$
and $\mu(\cdot)$ is a \textit{L\'evy measure}, that is, $\int_{\R}(1\wedge u^2)\,\mu(du)<\infty$ (the operator $(-\Delta+m^2)^s$ falls into this class). The function $e^{t\psi}$ is integrable under some conditions of the process $Z_t$ and hence the associated transition probability density $p_t(x)$ has the integral representation as the inverse Fourier transform of the characteristic function, $p_t(x)=\frac{1}{2\pi}\int_{\R}e^{-izx+t\psi(z)}\,dz$. It is possible to investigate this oscillatory integral, under the assumption that the characteristic exponent $\psi$ admits an \textit{analytic extension} to the complex plane. A usual assumption on the L\'evy measure is that it is exponentially integrable
$$
\int_{|y|\ge 1}e^{Cy}\mu(dy)<\infty, \quad \text{for all } C\in \R.
$$
 The latter assumption is satisfied, for instance, for a \textit{generalized tempered} L\'evy measure defined in terms of certain $\psi$ with super-exponential decay, i.e., $e^{Cu}\psi(u)\to 0$, $u\to \infty$, for all $C\in \R$, see for example \cite{KK,Sz}. Then we may find L\'evy processes where the transition probability density satisfies a ``weakly super-linear'' asymptotic behavior, namely, for constants $c_2<c_1$, there exists $y=y(c_1,c_2)$ such that
 $$
 \exp\big(-c_1 x \log^{\frac{\beta-1}{\beta}}\big(\frac{x}{t}\big)\big)\le p_t(x) \le  \exp\big(-c_2 x \log^{\frac{\beta-1}{\beta}}\big(\frac{x}{t}\big)\big), \quad x/t>y,
 $$
 where $\beta>1$, if the L\'evy measure $\mu$ satisfies certain exponential estimates (see \cite[(1.18)]{KK}).  It seems rather natural to look for non-local operators whose fundamental solutions have superlineal decay, alike the discrete case above mentioned, and that still have the analyticity properties mentioned above. We will explore this question in the future.

We finish this subsection with some further references on unique continuation. For elliptic nonlocal models, in \cite{R,RW,Spams,S} the authors use Carleman estimates. Other techniques are used: in \cite{BG,FF1,FF2} the so-called Almgren monotonicity formulas (see \cite{GL}) are used, and lower bounds and Runge approximation results for the fractional heat equation are proved in \cite{RS}.

\subsection{Importance of analyticity: revisiting the fractional Landis conjecture}

We make some remarks concerning the case $m=0$.  Let $N\ge 1$ and $s\in (0,1)$. Let $u: \mathbb{R}^N  \to \mathbb{R}$ be a solution to the equation
\begin{equation}
\label{landis}
(-\Delta)^su(x)=V(x) u(x), \quad x \in \mathbb{R}^N,
\end{equation}
with $V \in L^\infty(\mathbb{R}^N)$. In a recent work, R\"uland and Wang \cite{RW} proved that, for potentials with some a priori bounds, if a solution to Problem \eqref{landis} below decays at a rate $e^{-|x|^ {1+}}$, then this solution is trivial. On the other hand, for $s\in(1/4,1)$ and merely bounded non-differentiable potentials, if a solution decays at a rate $e^{-|x|^ {\alpha}}$ with $\alpha>4s/(4s-1)$, then this solution must again be trivial. We remark that when $s\to1$, $4s/(4s-1)\to4/3$, which is the optimal exponent for the standard Laplacian, see \cite{M}. In fact, in \cite[Theorem 3]{RW} they provide  a quantitative lower bound which leads to the mentioned unique continuation principle. Their result
motivates us to point out how another result on unique continuation, of qualitative nature, can be obtained as a consequence of lack of analyticity. 
 Assume that $u$, the solution to Problem \ref{landis}, decays exponentially fast at infinity, i.e.
\begin{equation}
\label{ux}
 |u(x)|\le e^{-c|x|^{1^+}}.
\end{equation}
Then it follows that $u\equiv 0$.  Indeed (for simplicity, we restrict ourselves to the one dimensional case), if we take the Fourier transform at both sides of equation \eqref{landis}
$$
|\xi|^{2s}\widehat{u}(\xi)=\widehat{F}(\xi), \qquad F:=Vu,
$$
we notice that the right hand side is analytic in $\C$ while the left hand side is not. This is justified as follows: observe that condition \eqref{ux} implies that $\widehat{u}(\xi)$ is analytic (the exponential decay of $u$ makes the Fourier transform of $u$ to be well defined, as well as its derivatives). Moreover, since $V$ is uniformly bounded then the right hand side term $F$ also decays exponentially fast at infinity and thus $\widehat{F}$ is analytic. Since $|\xi|^{2s}$ is not analytic at the origin we conclude that $\widehat u$ has to be identically zero.

We note that condition \eqref{ux} can be relaxed to an $L^2$ decay condition such as $\|u(\cdot) e^{c|\cdot|^{1^{+}}}\|_{L^2(\R^N)}<\infty$, which is sufficient to ensure the analyticity of $\widehat{u}$.     Observe also that the result can be seen as an optimal, qualitative Landis-type conjecture, valid for all $s\in (0,\infty)\setminus \mathbb{N}$. 
Finally we remark that only for $s=\frac{1}{2}$, Landis-type conjecture for \eqref{landis} is proved under the hypothesis of merely exponential decay, see \cite{KW}.

\begin{remark} Some comments are in order:
\begin{enumerate}
\item[(i)] In the case of  the parabolic problem $u_t(t,x) + (-\Delta+m^2)^s u(t,x) =V(t,x) u(t,x)$ ($m\ge 0$) the argument above does not work anymore since the exponential decay of $u$ assumed in \eqref{ux} does not need to be inherited by $u_t$.
\item[(ii)] The exponential decay is a sufficient condition that can not be improved with this technique. Less decay assumptions on $u$ were considered for instance by Frank, Lenzmann and Silvestre \cite{FLSi}, although their result concerns only radial solutions.
\end{enumerate}
\end{remark}

\subsection{Remarks and notations} We want to emphasize that what corresponds in fact to a diffusion problem is the following equation:
\begin{equation}
\label{eq:proRe2}
v_t(t,x)+((-\Delta + m^2)^s - m^{2s})v(t,x)=V(t,x)\, v(t,x), \quad x\in \RN, \,\,\, t>0,
\end{equation}
with data $v(0,x)=u_0(x)$, $x\in \RN$.  It is easy to check that mass is conserved  when $V=0$, i.e. $\int_{\RN} v(t,x)dx = \int_{\RN} u_0(x) dx$ for all $t>0$. Moreover,
\begin{equation}\label{transformation}
  u(t,x):= e^{-m^{2s}t} v(t,x)
\end{equation}is a solution to Problem \eqref{eq:proRel} with the same initial data. Throughout the paper we will work with Problem~\eqref{eq:proRel} . Most of the results can be reformulated in terms of the solution to Problem \eqref{eq:proRe2} via the transformation \eqref{transformation} or by simply adapting the definition of the operator adding the term $m^{2s}I$. For instance, the spatial behaviour for small times is the same for both $u$ and $v$.

We denote, for $m\ge 0$, the operator $$L_m:=(-\Delta + m^2 )
$$
(observe that $
L_0=(-\Delta)$). The main reason to work with $L^s_m:=(-\Delta + m^2)^s$ and not $\mathcal{R}^s_m := (-\Delta + m^2)^s - m^{2s}$ is that the composition law becomes simpler in the case of $L^s_m$, namely
$$
 L^s_m (L^s_m)=L^{2s}_m \qquad \text{ unlike }\qquad \mathcal{R}^s_m (\mathcal{R}^s_m)=\mathcal{R}^{2s}_m - 2m^{2s}\mathcal{R}^s_m.
$$

Along the paper we will use a fairly standard notation. We will just skip the variables $(t,x)$ of the functions in many of the instances e.g., we will sometimes use $u$ instead of $u(t,x)$. The complete expression will be used when relevant.

\section{The fractional relativistic operator. Definitions and properties}\label{Sect:Defn}

There are various equivalent definitions for $ L_m^s$, we state two of them below. The proof of the equivalence is given in Appendix \ref{equiv}. We will always consider real valued functions to avoid complex conjugates. We refer the reader to \cite{SKM, StSingular} for more information about
fractional powers of relativistic operators and Bessel potentials.  

\noindent (1) \textbf{Definition using the Fourier transform.}
For a function $f$ in the Schwartz class $\mathcal{S}$ we can define its Fourier transform as
$$
\widehat{f}(\xi)=\frac{1}{(2\pi)^{N/2}}\int_{\RN}f(x)e^{-ix\cdot \xi}\,dx,\quad \xi\in \RN.
$$
The inversion formula is given, for $f\in \mathcal{S}$, by
$$
\mathcal{F}^{-1}(f)(x)=\frac{1}{(2\pi)^{N/2}}\int_{\RN}f(\xi)e^{-i\xi\cdot x}\,d\xi,\quad x\in \RN.
$$
Let $0<s<1$,  $m\ge0$ and $f\in \mathcal{S}$. The operator $ L_m^s (f)$ is defined as a pseudo-differential operator
\begin{equation}\label{Defn:Fourier}
\widehat{L_m^s f  }(\xi) = (|\xi|^2 +m^ 2)^ s \widehat{f}(\xi)=(\xi\cdot \xi +m^ 2)^ s \widehat{f}(\xi),\quad \xi \in \RN.
\end{equation}

\noindent (2) \textbf{Definition via a singular integral.}
We first introduce some well known facts about the so called modified Bessel functions and Macdonald's functions, that will be useful later.
Let
$I_{\nu}(z)$ be the modified Bessel function of first kind given by the formula (see \cite[Chapter 5, Section 5.7]{Lebedev})
\begin{equation}
\label{eq:Inu}
I_{\nu}(z)=\sum_{k=0}^{\infty}\frac{(z/2)^{\nu+2k}}{\Gamma(k+1)\Gamma(k+\nu+1)}, \qquad |z|<\infty, \quad |\operatorname{arg} z|<\pi
\end{equation}
and let $ K_{\nu} $ be the Macdonald's function of order $\nu$ defined by (see also \cite[Chapter 5, Section 5.7]{Lebedev})
\begin{equation}
\label{eq:Knu}
K_{\nu}(z)=\frac{\pi}{2}\frac{I_{-\nu}(z)-I_{\nu}(z)}{\sin \nu\pi}, \qquad |\operatorname{arg} z|<\pi, \quad  \nu\neq 0,\pm1, \pm 2, \ldots
\end{equation}
and, for integral $\nu=n$, $K_n(z)=\lim_{\nu\to n}K_{\nu}(z)$, $n= 0,\pm1, \pm 2, \ldots$ From \eqref{eq:Inu} and \eqref{eq:Knu} it is clear that\begin{equation}
\label{eq:asymp0}
K_{\nu}(z)\sim\frac{\Gamma(\nu)}{2}\Big(\frac{z}{2}\Big)^{-\nu}, \quad (\Re \nu >0) \quad \text{ as } z\to 0^+.
\end{equation}
Moreover, it is well known (see \cite[Chapter 5, Section 5.11]{Lebedev}) that
\begin{equation}
\label{eq:asympInf}
K_{\nu}(z)=C e^{-z} z^{-1/2}+ \tilde{R}_{\nu}(z),\qquad
| \tilde{R}_{\nu}(z)|\le C_{\nu} e^{-z}z^{-3/2}, \quad |\operatorname{arg} z|\le\pi-\delta.
\end{equation}
We have the integral representation for the Macdonald's functions, also called Sommerfeld integral (see for instance \cite[p. 407]{NUV} or \cite[Chapter 5, (5.10.25)]{Lebedev}),
\begin{equation}
\label{formula:Bessel}
K_{\nu}(z)=2^{-\nu-1}z^\nu \int_0^\infty e^{-(t+\frac{1}{4t}z^2)} t^{-\nu-1}\,dt.
\end{equation}

Let $s\in (0,1)$, $m>0$ and $f$ be a real function with suitable decay at infinity, for instance $f\in C_b^2(\RN)$. Then $L_m^{s}(f)$ has a pointwise representation as
\begin{equation}
\label{eq:pointwise}
L_m^{s} f(x)=C_{N,s} m^{\frac{N+2s}{2}} \operatorname{P.V.} \int_{\mathbb{R}^N} \frac{f(x)-f(y)}{|x-y|^{\frac{N+2s}{2}}} K_{\frac{N+2s}{2}}(m|x-y|)\,dy + m^{2s} f(x), \quad
\forall x\in \RN,
\end{equation}
where $C_{N,s}$ is a normalization positive constant given by
\begin{equation}\label{constant}
C_{N,s}=-\frac{2^{1+s-N/2}}{ \pi^{\frac{N}{2}}\Gamma(-s)}.
\end{equation}
 Following \cite{FF2}, we  define the scalar product:
\begin{align*} \langle f,g\rangle_{\mathbb{H}^s_m (\RN)} &=\int_{\RN} \widehat{f}(\xi) \overline{\widehat{g}(\xi)}  (|\xi|^2+m^2)^s d\xi\\
&=  \frac{C_{N,s}}{2} m^{\frac{N+2s}{2}} \int_{\mathbb{R}^N} \int_{\mathbb{R}^N} \frac{\big(f(x)-f(y)\big)\big(g(x)-g(y)\big)}{|x-y|^{\frac{N+2s}{2}}}K_{\frac{N+2s}{2}}(m|x-y|)\,dy \, dx\\
&\quad  + m^{2s} \int_{\mathbb{R}^N} f(x)g(x) dx.
\end{align*}
The corresponding norm is
\begin{align*}\|f\|^2_{\mathbb{H}^s_m (\RN)} &= \int_{\RN}  |\widehat{f}(\xi)  | ^2  (|\xi|^2+m^2)^s d\xi\\
&=   \frac{C_{N,s}}{2}  m^{\frac{N+2s}{2}}  \int_{\mathbb{R}^N} \int_{\mathbb{R}^N} \frac{\big(f(x)-f(y)\big)^2}{|x-y|^{\frac{N+2s}{2}}}K_{\frac{N+2s}{2}}(m|x-y|)\,dy \, dx  + m^{2s} \|f\|_{L^2(\RN)}^2.\end{align*}
Then we define $\mathbb{H}^s_m (\RN)$ as the completion of  $C_0^{\infty}$ with the norm $\| \cdot\|_{\mathbb{H}^s_m (\RN)}$. If $m > 0$, $\mathbb{H}^s_m (\RN)$ coincides with the standard Sobolev space $\mathbb{H}^s(\RN)$. Thus, we can avoid writing the subscript $m$. 
Notice that, by Plancherel's theorem, 
$$\|L^{s}_m u\|_{L^2(\RN)}= \|u\|_{\mathbb{H}^{2s}_m(\RN)}.$$ Thus $L^s_m:\mathbb{H}^{2s}(\RN) \to  L^2(\RN)$ is a bounded linear operator.

\subsection{Leibniz rule and pointwise estimates}

Consider the operator
\begin{equation*}
H^s_m(f,g):=L_m^s(fg)-fL_m^s(g)-gL_m^s(f).
\end{equation*}
Indeed, $H_m^s$ is the remainder arising in the fractional Leibniz rule associated to our operator $L_m^s$. Moreover, $2H_m^s(f,g)$ is known in the literature as \emph{carr\'e du champ} operator, see definition and properties in
\cite[Subsection 1.4.2]{BGL}. Most of the properties we use are proved in \cite{BGL} for a general L\'evy process having an infinitesimal generator $L$. All the necessary information is stated in the proposition below, whose proof is just a consequence of the symmetry of the kernel, so we omit the details.

\begin{proposition}
\label{prop:leibniz}
Let $0<s<1$. For $f,g\in \mathbb{H}^s(\RN)  $     we have, for all $x\in \RN$,
\begin{equation*}
H^s_m(f,g)(x)=-C_{N,s} m^{\frac{N+2s}{2}}\int_{\mathbb{R}^N} \frac{\big(f(x)-f(y)\big)\big(g(x)-g(y)\big)}{|x-y|^{\frac{N+2s}{2}}}K_{\frac{N+2s}{2}}(m|x-y|)\,dy-m^{2s}f(x)g(x)
\end{equation*}
where $C_{N,s}$ is as in \eqref{constant}. In particular,
\begin{equation*}
H^s_m(f,f)(x)=-C_{N,s} m^{\frac{N+2s}{2}}\int_{\mathbb{R}^N} \frac{\big(f(x)-f(y)\big)^2} {|x-y|^{\frac{N+2s}{2}}} K_{\frac{N+2s}{2}}(m|x-y|)\,dy -m^{2s}f^2(x)\le 0.
\end{equation*}
Moreover,
\begin{equation}
\label{eq:CC}
L^s_m(f^2)(x) = 2 f(x) L^s_m(f)(x) +H^s_m(f,f)(x), \qquad L^s_m(f^2)(x) \le 2 f L^s_m(f)(x).
\end{equation}
\end{proposition}

\subsection{Construction of unbounded eigenfunctions}

In this subsection we will construct a special family of eigenfunctions of the operator $L^s_m$. First we prove some integral formulas involving the Bessel functions.

\begin{lemma}
\label{lemma:Bessel}
 Let  $N\ge 1$, $s\in  (0,1)$, $\lambda \in \RN$  with $|\lambda|<1$. Then
\begin{equation}\label{Bessel:identity1}
C_{N,s}\,  \int_{\RN} \frac{1-e^{  \lambda \cdot z}} {|z|^{\frac{N+2s}{2}}} K_{\frac{N+2s}{2}}(|z|)\, dz=\big( 1- |\lambda|^2\big)^s-1,
\end{equation}
where $C_{N,s}$ is given by \eqref{constant}.

Moreover, when $N-2s<1$ and $|\lambda|=1$, the following also holds
\begin{equation}\label{Bessel:identity2}
C_{N,s} \int_{\RN} \frac{1-e^{\lambda \cdot z}}{|z|^{\frac{N+2s}{2}}} K_{\frac{N+2s}{2}}(|z|)\,dz= -1.
\end{equation}

\end{lemma}
\begin{proof}
Observe that the  integrals are well defined (this can be checked by using the asymptotics of the Bessel function \eqref{eq:asymp0} and \eqref{eq:asympInf}). The restriction for $N$ and $s$  for the second identity comes from the integrability of the integral near the origin.

The identities follow by using the integral representation of the Bessel function \eqref{formula:Bessel}
and the identity
\begin{equation}
\label{formula:spower}
\gamma^s= \frac{1}{\Gamma(-s)} \int_0^\infty(e^{-t\gamma}-1)\frac{dt}{t^{1+s}}, \quad \gamma>0, \quad 0<s<1.
\end{equation}
For the proof of  \eqref{Bessel:identity1} observe that
\begin{align*}
C_{N,s} \int_{\RN} \frac{1-e^{ \lambda\cdot z}} {|z|^{\frac{N+2s}{2}}} K_{\frac{N+2s}{2}}(|z|)\, dz &= C_{N,s} 2^ { -\frac{N+2s}{2}-1} \,  \int_{\RN} ( 1-e^{ \lambda\cdot z} )
\int_0^\infty e^{\big(-t-\frac{|z|^2}{4t}  \big)} \frac{dt}{t^{\frac{N+2s}{2}+1}} dz \\
&=C_{N,s} 2^ { -\frac{N+2s}{2}-1} \, \int_0^\infty  \left[\int_{\RN}   ( 1-e^{ \lambda\cdot z} ) e^{-\frac{|z|^2}{4t}  } dz \right]   e^{-t}   \frac{dt}{t^{\frac{N+2s}{2}+1}} .
\end{align*}
Notice that the integral in $z$, after changing variables to $z=2\sqrt{t} y$, equals
\begin{align*}
 \int_{\RN}   ( 1-e^{ 2 \sqrt{t} \lambda \cdot  y } ) e^{-|y|^2  } (2\sqrt{t})^N dy &= 2^{N}t^{\frac{N}{2}} \Big ( \int_{\RN}    e^{-|y|^2  }dy -  \int_{\RN}   e^{ 2 \sqrt{t}
 \lambda\cdot  y -|y|^2  } dy  \Big) \\
&= 2^{N}t^{\frac{N}{2}}\Big( \sqrt{\pi}^N -  e^{ t |\lambda|^2 } \int_{\RN}   e^{ - (y -    \sqrt{t} \lambda)^2  } dy \Big) \\
&= 2^{N}t^{\frac{N}{2}} \sqrt{\pi}^N  \big(1 -  e^{ t |\lambda|^2 } \big).
\end{align*}
Then using the explicit form of the constant given in \eqref{constant} we obtain that
\begin{align*}
C_{N,s} \int_{\RN} \frac{1-e^{ \lambda \cdot z}} {|z|^{\frac{N+2s}{2}}} K_{\frac{N+2s}{2}}(|z|)\, dz& =  C_{N,s} 2^ { -\frac{N+2s}{2}-1 +N}   \sqrt{\pi}^N  \, \int_0^\infty
t^{\frac{N}{2}}  \Big (1 -  e^{ t |\lambda|^2 }  \Big)  e^{-t}   \frac{dt}{t^{\frac{N+2s}{2}+1}} \\
&=C_{N,s} 2^ {  \frac{N-2s}{2}-1 }   \sqrt{\pi}^N  \, \int_0^\infty   \Big(e^{-t} -1+1-  e^{  - t (1- |\lambda|^2 ) }  \Big)    \frac{dt}{t^{s+1}} \\
&=C_{N,s} 2^ {  \frac{N-2s}{2}-1 }   \sqrt{\pi}^N  \,  \Gamma(-s) \Big( 1 -    \big(1- |\lambda|^2 \big)^s    \Big)  =  \big(1- |\lambda|^2 \big)^s-1.
\end{align*}
For \eqref{Bessel:identity2} the proof is the same, except for the  last integral is $ \int_0^\infty   \big (e^{-t} - 1  \big)    \frac{dt}{t^{s+1}}  dt$, which equals $\Gamma(-s)$ according
to \eqref{formula:spower}.

\end{proof}

\begin{proposition}
\label{prop:analy}
Let  $N\ge 1$, $s\in  [0,1]$, $\lambda \in \RN$  with $|\lambda|<m$. Then
\begin{equation}\label{formula:eigenfunction}
 L_m^s e^{\lambda \cdot x}=(-|\lambda|^2+m^2)^s e^{\lambda\cdot x}, \quad \text{a.e. } x\in \mathbb{R}^N.
\end{equation}
Moreover, when $N-2s<1$ and  $|\lambda|=m$ then
\begin{equation*}
 L_m^s e^{\lambda \cdot x}=0, \quad \text{a.e. } x\in \mathbb{R}^N.
 \end{equation*}
\end{proposition}
\begin{proof} First, observe that the cases $s=0$ and $s=1$ follow trivially. Let $N\ge 1$, $s\in (0,1)$ and  $\lambda \in \RN$ as in the hypothesis. Then  for $x\in \RN $
$$
 L_m^s(e^{\lambda \cdot (\cdot)})(x) - m^{2s}e^{\lambda \cdot x} = C_{N,s} \,  m^{ 2s }   e^{\lambda \cdot x }  \int_{\RN} \frac{1-e^{\frac{1}{m} \lambda \cdot z }}{|z |^{\frac{N+2s}{2}}} K_{\frac{N+2s}{2}}(|z|) \,dz .
$$
The integral is well defined, as proved in Lemma \ref{lemma:Bessel} with $\lambda/m$ as the corresponding parameter. Now we use Lemma \ref{lemma:Bessel} with $\lambda/m \in \RN$,
$|\lambda|/m \le  1$ and the result follows in each of the cases.
\end{proof}
\begin{remark}
\label{rem:analyti}
The identity \eqref{formula:eigenfunction} cannot be extended to $|\lambda|>m$ because the function $\C\ni z\to (z^2 + m^2) ^s$ is not well defined in $\{ it :t > m\}$.
\end{remark}

\section{The heat equation for the fractional relativistic operator}\label{Sect:Heat}

We devote this section to the linear heat equation
\begin{equation}\label{eq:heat}
\begin{cases}
   u_t(t,x) + (-\Delta+m^2)^{s}u(t,x)=0, &\quad x\in \RN, \,\,\, t>0, \\
u(0,x)=u_0(x), &\quad x\in \RN.
    \end{cases}
    \end{equation}

\subsection{Fundamental solution}\label{Fundamental:Solution}
The fundamental solution $K_t^s$ for the heat equation involving $L_m^s$ is defined via the Fourier transform as
$$
\widehat{K}_t^s(\xi)=e^{-t(|\xi|^2+m^2)^s}.
$$
This will correspond to the probability density function of the associated stable relativistic process. Estimates for the fundamental solution are well-known and can be found in \cite[Subsection 6.4]{Sz} (see also \cite[Theorem 1.2]{CKK} and \cite[Theorem 4.1]{CKS}). For our purposes, we will emphasize that $K^s_t(x)$ is a smooth function for $t>0$ and $x\in \mathbb{R}^N$, since its Fourier transform decays exponentially. Moreover, 
$K_t^s$ has an exponential decay in $|x|$ for small times:
$$
K_t^s(x)\sim c(t) |x|^{-N-2s} e^{ -\widetilde{c}\,m|x|}, \quad  x \in \R^N,
$$
with a positive constant $\widetilde{c}$ independent of time. 

\subsubsection{Integral representation}

The following subordination formula is shown in \cite[(7)]{Ry}, see also \cite{GrRy},
\begin{equation}
\label{eq:Ksubo}
K_t^s(x)=\int_0^{\infty}\Theta^s_t(\rho)e^{-m^2\rho}\frac{e^{-\frac{|x|^2}{4\rho}}}{(4\pi \rho)^{N/2}} \,d\rho.
\end{equation}
Here, $\Theta_t^s(\rho)$, $\rho>0$, is the density function of the $s$-stable process whose Laplace transform is $e^{-t\lambda^{s}}$. In the case $s=1/2$, $K_t^{1/2}$ has an explicit expression, as stated in the Introduction. We compute the following weighted $L^1$ norm of $K^s_t$.

\begin{lemma}
Let $0<s<1$. We have, for all $|\lambda| \le m$,
\begin{equation}\label{weightedL1:K}
\|e^{\lambda \cdot (\cdot)}  K_t^s(\cdot) \|_{L^1(\RN)} = e^{- (m^2 - |\lambda|^2)^s\, t},\quad t>0.
\end{equation}
\end{lemma}
\begin{proof}
In \cite[p. 3]{Ry}, Ryznar defines the probability density function
$$
\Theta_t^s(\rho,m)=e^{-m^2\rho+m^{2s}t}\Theta_t^s(\rho),\quad \rho>0,
$$
so in particular this means that
\begin{equation}
\label{Ryznar}
\int_0^{\infty}e^{-m^2\rho}\Theta_t^s(\rho)\,d\rho=e^{-tm^{2s}}.
\end{equation}
Now, by \eqref{eq:Ksubo} and Fubini's Theorem,
\begin{equation}
\label{fubin}
\int_{\RN}e^{\lambda\cdot x}\int_0^{\infty}\Theta^s_t(\rho)e^{-m^2\rho}\frac{e^{-\frac{|x|^2}{4\rho}}}{(4\pi \rho)^{N/2}} \,d\rho\,dx=\int_0^{\infty}\Theta^s_t(\rho)e^{-m^2\rho}(4\pi \rho)^{-N/2}\int_{\R}e^{\lambda\cdot x}e^{-\frac{|x|^2}{4\rho}}\,dx \,d\rho.
\end{equation}
Observe that
$$
\int_{\RN}e^{\lambda\cdot x}e^{-\frac{|x|^2}{4\rho}}\,dx=\prod_{i=1}^N\int_{\R}e^{\lambda_i x_i}e^{-\frac{x_i^2}{4\rho}}\,dx_i=e^{|\lambda|^2\rho}\prod_{i=1}^N\int_{\RN}e^{-\big(\frac{x_i}{2\sqrt{\rho}}-\lambda_i\sqrt{\rho}\big)^2}\,dx_i=e^{|\lambda|^2\rho}2^N(\pi\rho)^{N/2}.
$$
Therefore, \eqref{fubin} equals
$$
\int_0^{\infty}\Theta^s_t(\rho)e^{-\rho(m^2-|\lambda|^2)}\,d\rho=e^{-t(m^{2}-|\lambda|^2)^2},
$$
where the equality follows from \eqref{Ryznar}.
\end{proof}

\subsection{Energy estimates}

Let $u$ be a  solution to Problem \eqref{eq:heat} with   $u_0 \in L^1(\RN)\cap L^\infty(\RN)$. Thus $u$ is obtained directly from the fundamental solution and the data
\begin{equation}\label{convolution}
u(t,x)=u_0(x) \ast K^s_t(x), \quad x\in \RN, \,\,\, t>0.
\end{equation}
In view of the smoothness of $ K^s_t(x)$, it follows that $u$ is a smooth function for $t>0$ and $x\in \RN$. Thus, the estimates we give in the following theorem are justified from the regularity point of view.

\begin{proposition}\label{EnergyEstim}
Let $0<s<1$. Let $u$ be the  solution to Problem \eqref{eq:heat} with   $u_0 \in L^1(\RN)\cap L^\infty(\RN)$. 
 Then:
\begin{itemize}
\item (Decay of total mass) We have
\begin{equation}
\label{energy0}
\int_{\RN}u(t,x)\,dx=e^{-m^{2s}t}\int_{\RN}u_0(x)\,dx, \quad t>0.
\end{equation}
\item (Energy estimate) 
For all $0<t<T$ we have
\begin{equation}\label{energy1}
2\int_0^t\int_{\RN} |L^{s/2}_m u(\tau,x)|^2 \,dx \,d\tau +  \int_{\RN} u^2(t,x) \,dx = \int_{\RN} u_0^2(x)\, dx.
\end{equation}
\item (Decay of weighted $L^2$ norm) Assume $u_0 \in L^2(e^{\lambda \cdot x},\RN)$. Then for all $0<t$ and $|\lambda| \le 2m$ we have
\begin{equation}\label{energy:weighted}
\int_{\RN}  u^2 (t,x) e^{\lambda\cdot x} \,dx \le  e^{- (m^2-|\lambda|^2/4)^s t } \int_{\RN}  u_0^2 (x) e^{\lambda\cdot x} \,dx, \quad t\ge 0.
\end{equation}
\end{itemize}

\end{proposition}
\begin{proof}
The first identity follows from \eqref{transformation} and mass conservation for Problem \eqref{eq:proRe2}. On the other hand, \eqref{energy1} follows from the fact that, formally,
$$
\frac{d}{dt}\int_{\RN}  u^2(t,x)\, dx = -2 \int_{\RN}  u L^s_mu(t,x) \,dx = -2 \int_{\RN}  |L^{s/2}_m u(t,x)|^2 \,dx.
$$
Notice that the identities in \eqref{energy0} and \eqref{energy1} are typical energy estimates for diffusion equations. For the inequality \eqref{energy:weighted} we proceed as follows. Let $\lambda \in \RN$ with $|\lambda|\le 2m$. Then, by using the representation \eqref{convolution} and \eqref{weightedL1:K} with $\lambda/2$, we derive that
  \begin{align*}
 \int_{\RN} u^2(t,x)  e^{\lambda \cdot x} \,dx   &  =   \int_{\RN} e^{\lambda \cdot x} ((K_t^{s} \ast u_0)(x))^2 \,dx \\
     & = \int_{\RN} \Big(  \int_{\RN}  e^{\lambda/2 \cdot (x-y)} K_t^s(x-y) e^{\lambda/2 \cdot y}u_0(y) \, dy    \Big)^2\,dx   \\[2mm]
     &=\|e^{\lambda/2 \cdot (\cdot)} K_t^s(\cdot) \ast (e^{\lambda/2 \cdot (\cdot)}u_0(\cdot)) \|_{L^2(\RN)}^2 \\[2mm]
     &\le \|e^{\lambda/2 \cdot (\cdot)} K_t^s(\cdot)\|_{L^1(\RN)}^2 \| (e^{\lambda/2 \cdot (\cdot)}u_0(\cdot)) \|_{L^2(\RN)}^2  \\[2mm]
     &= e^{-(m^2-|\lambda|^2/4)^s t } \,  \| (e^{\lambda/2 \cdot (\cdot)}u_0(\cdot)) \|_{L^2(\RN)},
  \end{align*}
  as desired.
\end{proof}

\subsection{Construction of separate variable solutions}

With the tools we have so far, it is easy to construct explicit separate variables solutions  $\omega: (0,T)\times \R^N  \to \R$ to the equation
\begin{equation}\label{FHE}
\partial_t\omega(t,x)=- L^s_m \omega (t,x),\quad x\in \RN, \,\,\, t>0.
\end{equation}
We are interested in spatial increasing solutions. By Proposition \ref{prop:analy} we obtain that
\begin{equation*}
\omega_\lambda (t,x)=e^ {-(-|\lambda|^ 2+m^ 2)^ s \, t} e^ {\lambda \cdot x},\quad x\in \RN, \,\,\, t>0,
\end{equation*}
verifies the equation \eqref{FHE} for every $\lambda \in \RN$ with $|\lambda|<m$.  Moreover
$$L^s_m \omega (t,x)= (-|\lambda|^ 2+m^ 2)^ s \omega (t,x),\quad x\in \RN, \,\,\, t>0.
$$


\subsection{Log-convexity of the weighted functional for the linear heat equation}
\label{subseclog}

Let $u$ be the solution to Problem \eqref{eq:heat} with $u_0 \in L^2(e^{\lambda \cdot x},\RN)$ and let $\mathcal{H}(t):=\int_{\RN}e^{\lambda \cdot x}u^2(t,x)\,dx$, which is well defined for $|\lambda|\le 2m$ according to \eqref{energy:weighted}. Moreover,
$$
\mathcal{H}(t)=\int_{\RN}|e^{\lambda/2 \cdot x}u(t,x)|^2\,dx=\int_{\RN}\big|\widehat{u}\big(\xi+\frac{i\lambda}{2}\big)\big|^2\,d\xi=\int_{\RN}e^{-2t((\xi +\frac{i\lambda}{2}) ^2+m^{2})^s} \big|\widehat{u_0}\big(\xi+\frac{i\lambda}{2}\big)\big|^2\,d\xi.
$$
\begin{theorem}
Let $0<s<1$. The functional $\mathcal{H}$ is logarithmically convex. In particular
$$\mathcal{H}(t) \le \mathcal{H}(0)^{1-t} \mathcal{H}(1)^t, \quad  t\in [0,1].$$
\end{theorem}
\begin{proof}
Starting from \eqref{convolution}, we use the Fourier representation of the solution $u$. Then, the functional $\mathcal{H}(t)$ can be written as follows
\begin{align*}
\mathcal{H}(t)&=\| e^{\lambda/2 \cdot (\cdot)} (K^s_t \ast u_0) (\cdot)  \|^2_{L^2(\RN)} =\Big \| e^{-t((\cdot +\frac{i\lambda}{2}) ^2+m^{2})^s} \widehat{u_0}\Big(\cdot+\frac{i\lambda}{2}\Big) \Big \|^2_{L^2(\RN)}.
\end{align*}
Since $$(\log(\mathcal{H}))''=\frac{\ddot{\mathcal{H}}(t)  \mathcal{H}(t)-\dot{\mathcal{H}}(t)^2 }{(\mathcal{H}(t))^2}$$
we only need to check that the numerator is positive. Indeed, by using $\widehat{\frac{d}{dt}K^s_t}(\xi)=(\xi\cdot \xi+m^{2})^s e^{-t\, (\xi\cdot \xi+m^{2})^s}$, we have
\begin{align*}
(\dot{\mathcal{H}}(t))^2&= \Big\|\Big(\big(\cdot +\frac{i\lambda}{2}\big) ^2+m^{2}\Big)^s e^{-t\, ((\cdot +\frac{i\lambda}{2}) ^2+m^{2})^s} \widehat{u_0}\Big(\cdot+\frac{i\lambda}{2}\Big)  \Big\|^4_{L^2(\RN)}\\
&\le\Big\|\Big(\big(\cdot +\frac{i\lambda}{2}\big) ^2+m^{2}\Big)^{2s} e^{-t\, ((\cdot +\frac{i\lambda}{2}) ^2+m^{2})^s} \widehat{u_0}\Big(\cdot+\frac{i\lambda}{2}\Big) \Big \|^2_{L^2(\RN)} \Big\|e^{-t\, ((\cdot+\frac{i\lambda}{2}) ^2+m^{2})^s} \widehat{u_0}\Big(\cdot+\frac{i\lambda}{2}\Big) \Big \|^2_{L
^2(\RN)}\\
&=\ddot{\mathcal{H}}(t)  \mathcal{H}(t)
\end{align*}
and we conclude the proof.
\end{proof}

\begin{remark}
The logarithmic convexity is a strong tool that might lead to an uncertainty principle result for the corresponding equation, like in \cite{EKPV-JEMS}. However, the method developed in \cite{EKPV-JEMS} in order to prove uncertainty principles does not work here: the reason is that the decay of $u_0$ given by $\mathcal{H}(0)<\infty$ does not lead to better decay for $u(t,x)$, $t>0$.  This can be immediately seen from the definition $\mathcal{H}(t):= \int_{\R^N} e^{\lambda \cdot x} u^2(t,x) \,dx$, where the space decay is always the same ($\lambda$ is constant).
 In particular, if we look at the fundamental solution, it always has the same spatial decay $e^{-m|x|}$, thus it does not improve with time. This is  in big contrast to what happens with self-similar processes (for instance, as in \cite{EKPV-JEMS} dealing with the classical heat equation).
\end{remark}

\section{The heat equation with potential}\label{Sect:HeatPotential}

We devote this section to the study of (any sign) solutions to Problem \eqref{eq:proRel}. We assume that $V\in L^\infty([0,\infty)\times \RN)$.
\begin{defn}
Let $u_0\in \mathbb{H}^{2s} (\RN)$ and $T>0$ or $T=\infty$.
A \emph{mild solution} of Problem \eqref{eq:proRel} is a function $u\in C([0,T]:\mathbb{H}^{2s}(\RN) )$ which satisfies, for a.e. $(t,x)\in [0,T]\times \RN$,
\begin{equation*}
  u(t,x)=(K_t^{s} \ast u_0)(x)  + \int_0^t \big(K_{t-\tau}^{s} \ast  (V(\tau,\cdot) u(\tau,\cdot))\big)(x) d\tau.
\end{equation*}
\end{defn}
The diffusion operator $L^s_m: \mathbb{H}^{2s}_m(\RN) \to L^2(\RN)$ is a bounded linear operator. The right hand side $F: [0,\infty) \times \mathbb{H}^{2s}_m(\RN) \to \mathbb{H}^{2s}_m(\RN)$ with $F(t,x):=V(t,x) u(t,x)$ is Lipschitz for $V\in L^\infty([0,T]\times \RN).$
The hypothesis of  \cite[Thm. 1.2, Ch. 6.1]{P} is satisfied, thus there exists a mild solution to the initial value Problem \eqref{eq:proRel} for some $T>0$. Moreover, given the assumption on the potential, $u$ is in fact a strong solution: $u_t \in L^1([0,T]:\mathbb{H}^{2s}_m(\RN))   $, see \cite[Thm. 1.6, Ch. 6.1]{P}.

Since we will work with exponentially decaying solutions, we  prove here that, given $u_0 \in \mathbb{H}^{2s}_m (\RN) \cap L^{2}(e^{\lambda \cdot x} \,dx)$, then, the corresponding strong solution $u$ to Problem \eqref{eq:proRel} satisfies $u(t,\cdot) \in  L^{2}(e^{\lambda \cdot x} \,dx)$ for all $t>0.$ 

\begin{proposition}\label{Prop:weightedL2potential}  Let $u_0 \in \mathbb{H}^{2s}_m (\RN) \cap L^{2}(e^{\lambda \cdot x} \,dx)$. Then, the corresponding strong solution $u$ to Problem~\eqref{eq:proRel} satisfies $u(t,\cdot) \in  L^{2}(e^{\lambda \cdot x} \,dx)$ for all $t \in  [0,T].$ 
\end{proposition}
\begin{proof}
 The $L^2$ energy estimate can be proved for $0\le t   \le T$, integrating by parts:
$$\|u(t,\cdot)\|_{L^2(\RN)} \le e^{t \|V \|_{L^\infty([0,T]\times \RN)} } \|u_0\|_{L^2(\RN)}.
$$

Using \eqref{energy:weighted}, we have
\begin{align*}
& \| e^{\lambda/2 \cdot x}u(t,x)\|_{L^2(\RN)}
\le   \Big\| e^{\lambda/2 \cdot x}  ( K_t^s\ast u_0 ) \Big\|_{L^2(\RN)} +   \Big\|  e^{\lambda/2 \cdot x} \, \int_0^t K_{t-\tau}^{s} \ast (V(\tau,\cdot) u(\tau,\cdot)) d\tau \Big\|_{L^2(\RN)}  \\
 &\le e^{- \frac{(m^2-|\lambda|^2/4)^s}{2} \,  t }   \| e^{\lambda/2 \cdot (\cdot)}u_0\|_{L^2(\RN)} \\
 &+ \Big( \int_{\RN}  \Big(\int_0^t (e^{\lambda/2 \cdot (\cdot)}K_{t-\tau}^{s}(\cdot)) \ast  \big(V(\tau,\cdot) u(\tau,\cdot)\big)(x) \,   d\tau  \Big)^2  \,dx \Big)^{1/2}\\
&\le e^{- \frac{(m^2-|\lambda|^2/4)^s}{2} \,  t }   \| e^{\lambda/2 \cdot (\cdot)}u_0\|_{L^2(\RN)} + \| V\|_{L^\infty} \int_0^t \|e^{\lambda \cdot (\cdot)} K_{t-\tau}^s(\cdot)\|_{L^1(\RN)} \|u(\tau,\cdot)\|_{L^2(\RN)} \,d\tau\\
&\le e^{- \frac{(m^2-|\lambda|^2/4)^s}{2} \,  t }   \| e^{\lambda/2 \cdot (\cdot)}u_0\|_{L^2(\RN)} + \| V\|_{L^\infty} \sup_{t\in [0,T]}\|  u(t,\cdot) \|_{L^2(\RN)}^2 \int_0^t \|e^{\lambda \cdot (\cdot)} K_{t-\tau}^s(\cdot)\|_{L^1(\RN)}  \,d\tau \\
&= e^{- \frac{(m^2-|\lambda|^2/4)^s}{2} \,  t }   \| e^{\lambda/2 \cdot (\cdot)}u_0\|_{L^2(\RN)} +\| V\|_{L^\infty} \sup_{t\in [0,T]}\|  u(t,\cdot) \|_{L^2(\RN)}^2 \int_0^t e^{-(m^2-|\lambda|^2)^s \,(t-\tau)}  \,d\tau < +\infty.
 \end{align*}
The result follows.
\end{proof}

\begin{remark}
Observe that, in the proof of Proposition \ref{Prop:weightedL2potential} it is enough to assume only $u_0\in L^2(\RN)$. Nevertheless, since we will need  $u_0\in L^{2}(e^{\lambda \cdot x} \,dx)$ later, we have decided to state the more restrictive condition in the proposition. 
\end{remark}

\subsection{Backward unique continuation for the heat equation with potential}

\begin{theorem}[Backward unique continuation]
Let $N\ge 1$, $s\in (0,1)$ and $m\ge 0$.  Let $u$ be a solution to Problem \eqref{eq:proRel} with initial data $u_0 \in L^2( \RN)$ and $V\in L^\infty( 0,T \times \RN)$. Assume that $u(T,\cdot)=0$ for some time $T>0$.
Then $u\equiv 0$.
\end{theorem}
\begin{proof}
Let $\mathcal{H}(t):=\int_{\RN}u^2(t,x)\,dx$.
In case the potential is $V=0$  we have
$$
\dot{\mathcal{H}}(t)=\langle u, u_t\rangle + \langle u_t,u\rangle  = 2 \Re \langle u, L^s_m u\rangle  = 2 \|L^{s/2}_mu\|_{L^2(\RN)}^2
$$
and
$$
\ddot{\mathcal{H}}(t) = 2 \Re \langle u, u_{tt}\rangle  + 2\langle u_t,u_t\rangle = 4 \|L^s_m u\|_{L^2(\RN)}^2. $$
Then
$ ( \dot{\mathcal{H}}(t))^2 \le 4 \langle u,u\rangle  \cdot \langle L^s_m u,L^s_m u\rangle $ and thus $\mathcal{H}(t)$ is logarithmically convex. Hence
$$
\mathcal{H}(t)\le \mathcal{H}(0)^{\theta} \mathcal{H}(T)^{1-\theta}, \quad \theta\in [0,1].
$$
When the potential is non trivial, the functional $\mathcal{H}$ is still logarithmically convex, since the operator $L^s_m$ is symmetric. According to  \cite[Lemma 2, p. 6]{EKPV-JEMS}  there exists a constant $N$
such that
$$
\mathcal{H}(t)\le e^{N( \|V\|_\infty +\|V\|_\infty^2) } \mathcal{H}(0)^{t} \mathcal{H}(1)^{1-t}, \quad  t\in [0,1].
$$ Thus, up to scaling in time, if $u(T)\equiv 0$ then $u\equiv 0$.
\end{proof}


\section{Carleman inequality for the parabolic operator with linear exponential weight}\label{Sect:Carleman1}


This section is devoted to the study of convexity estimates for an exponential weighted norm of solutions $u:(0,T]\times \R^N\to \R$  to the initial value problem
\begin{equation}\label{Eq:withF}
\begin{cases}
u_t(t,x)+(-\Delta + m^2)^s u(t,x)=F(t,x) &\quad x\in \RN, \,\,\, t>0, \\
u(0,x)=u_0(x), &\quad x\in \RN
\end{cases}
\end{equation}
for some $T>0$. In particular, when $F=Vu$ we are reduced to the Problem \eqref{eq:proRel}. We consider the functional
\begin{equation}
\label{eq:Ht}
\mathcal{H}(t):=\int_{\RN} \omega(t,x) u^2(t,x) \,dx, \quad \omega(t,x)=e^{At+\lambda \cdot x},  \quad  t\in [0,T],
\end{equation}
for $A\in \R$, $\lambda\in \R^N$.
Note that, by Proposition  \ref{Prop:weightedL2potential}, we have $\mathcal{H}(t)<\infty$ for $|\lambda|\le 2m$.

\subsection{Persistence of the spatial decay: monotonicity of the energy functional}


In what follows we prove that if the initial data decays at least exponentially fast in space, then the solution $u(t,x)$ will have a similar decay at every positive time $t>0$.
This will be a consequence of the monotonicity of the functional~$\mathcal{H}(t)$.

\begin{proposition}
\label{prop:monot}
Let $0<s<1$, $m>0$ and let $u_0 \in   \mathbb{H}^{2s}(\RN) \cap   L^2(e^{\lambda\cdot x}\, dx)$ for some $\lambda \in \RN$ with $|\lambda|\le m$. Assume that $F(t,x) \in L^2(0,T:L^2(e^{\lambda\cdot x}\, dx))$. Let $u$ be a  strong   solution to the initial value problem
\eqref{Eq:withF}.  
 Let $\omega(t,x)$ be as in~\eqref{eq:Ht}. Then, for all $t \in [0,T]$,
\begin{multline}\label{persistence} \int_{\RN} \omega(t,x) u^2(t,x) dx +  \int_0^t  \int_{\RN} \omega(t,x)  (-H^s_m(u,u)) \,dx \,dt  \\ \nonumber
\le e^{(A - ( -|\lambda|^2 +m^ 2 )^ s)t} \int_{\RN} \omega(t,x) u_0^2(x) \,dx   + e^{(A - ( -|\lambda|^2 +m^ 2 )^ s)t} \int_0 ^t \int_{\RN} \omega(t,x)  F^2(t,x) \,dx\,dt.  \end{multline}
\end{proposition}
\begin{proof}
Let $\mathcal{H}(t)$ be defined as in \eqref{eq:Ht}.
We have that
\begin{align*}
\dot{\mathcal{H}}(t) &=   \int_{\RN} (\omega_t- L^s_m  (\omega)) u^2 \,dx   + \int_{\RN} \omega H^s_m(u,u)\, dx +  2 \int_{\RN}  \omega u F \,dx \\
&= [A - ( -|\lambda|^2 +m^ 2 )^ s]  \int_{\RN}   \omega  u^2 \,dx   + \int_{\RN} \omega H^s_m(u,u)\, dx +  2 \int_{\RN} \omega u F\, dx.
\end{align*}
Let $a:= [A-  ( -|\lambda|^2 +m^ 2 )^ s ]$. For sufficiently negative $A$, the coefficient $a$ will be negative and thus we could ignore the term $a \int_{\RN}   \omega  u^2 \,dx $.
However, we keep this term to avoid imposing conditions on the parameter $A$ in this proposition. Thus we have
$$  \dot{\mathcal{H}}(t) - a \mathcal{H}(t)\le \int_{\RN}  \omega  H^s_m(u,u) \,dx+ \int_{\RN}  \omega  F^2 \,dx, $$
that can be rewritten as
$$\frac{d}{dt}\big( e^{-at} \mathcal{H}(t)\big) \le e^{-at} \int_{\RN}  \omega  H^s_m(u,u) \,dx+ e^{-at} \int_{\RN}  \omega  F^2 \,dx. $$
We integrate from $t_1$ to $t_2$ in time and therefore
$$ \mathcal{H}(t_2) +  e^{at_2}\int_{t_1}^{t_2} e^{-a\tau} \int_{\RN} \omega  (-H^s_m(u,u))\, dx \,dt \le e^{a(t_2- t_1)}\mathcal{H}(t_1)  +e^{at_2} \int_{t_1} ^{t_2}e^{-a\tau}
\int_{\RN} \omega  F^2 \,dx\,dt.$$
This implies that, for all $0\le t_1 <t_2\le 1$,
\begin{equation}\label{persistence:t1t2}
 \mathcal{H}(t_2) +  \int_{t_1}^{t_2} \int_{\RN} \omega  (-H^s_m(u,u)) \,dx\, dt \le e^{at_2}\mathcal{H}(t_1)  +e^{at_2} \int_{t_1} ^{t_2}  \int_{\RN} \omega  F^2 \,dx\,dt.
\end{equation}
The conclusion follows just by taking $t_1=0$ and $t_2=t$.
\end{proof}
\begin{remark}
In particular, if we take $t_2=1$ in \eqref{persistence:t1t2} and renaming $t_1$ into $t$, we obtain that, for all $t \in(0,1)$,
\begin{equation}
\label{persistence:t:1}
\mathcal{H}(1)+\int_{t}^{1} \int_{\RN} \omega  (-H^s_m(u,u)) \,dx\, dt \le e^{a}\mathcal{H}(t)  +e^{a} \int_{t} ^{1}  \int_{\RN} \omega  F^2 \,dx\,dt.
\end{equation}
\end{remark}

\subsection{Convexity arguments}

In this subsection we will consider a similar weight as in Subsection \ref{subseclog}, but with a correction in time needed in order to absorb the effects of the potential. We will prove a convexity result related to it, that will be the key point to get
a Carleman inequality for $\partial_t+L_m^s$ in Subsection \ref{subsec:Carelx}. Note that the proof carried out in Subsection \ref{subseclog} is not valid anymore due to the presence of the potential.

Let  $\omega(t,x)$ be defined as in \eqref{eq:Ht}, where $A$ is a constant to be chosen later.
Let
\begin{equation}
\label{eq:D}
D(t):=\int_{\RN}  \omega_t u^2 \, dx - 2\int_{\RN}  \omega uL^s_m u\, dx =\int_{\RN}  (\omega_t  - L^s_m \omega)u^2\,dx+\int_{\RN}\omega H_m^s(u,u)\,dx
\end{equation}
where the second equality follows easily from \eqref{eq:CC}. Actually,  \eqref{eq:D} is a formal definition, for any $\omega$. Recall also the definition of $\mathcal{H}(t)$ in  \eqref{eq:Ht}.
 We will prove the following.
\begin{proposition}
\label{prop:lower}
 Let $N\ge 1$, $s\in (0,1/2)$, $m>0$ and let $u_0 \in   \mathbb{H}^{2s}(\RN) \cap   L^2(e^{\lambda\cdot x}\, dx)$ for some $\lambda \in \RN$ with $|\lambda|\le m$. Assume that $F(t,x) \in L^2(0,T:L^2(e^{\lambda\cdot x}\, dx))$. Let $u$ be a  strong   solution to the initial value problem
\eqref{Eq:withF}  such that $u^2\in \operatorname{Dom}(L_m^{2s})$.  Let  $\omega(t,x)$ be as in \eqref{eq:Ht}.
For $A+m^{2s}\le0$ sufficiently small (that is, $|A|$ sufficiently large) satisfying
\begin{equation}\label{condA:2}
\frac14\big((-|\lambda|^2+m^2)^s-A\big)^2 \ge C_2(N,s) m^{4s},
\end{equation}
where $C_2(N,s)$ is a positive constant depending only on $N$ and $s$, we obtain the following lower bound
\begin{align}\label{Ddot:lowerbound}
\notag\dot{D}(t) &\ge
\frac34\big((-|\lambda|^2+m^2)^s-A\big)^2\mathcal{H}(t)  -C_1(N,s)\int_{\RN} \omega F^2  \,dx   + 2  \int_{\RN} \omega ( u_t)^2  \,dx \\
& \quad + (A+m^{2s}) \int_{\RN} H^s_m(u,u) \omega \, dx-\int_{\RN}\omega H^{2s}_m(u,u) \,dx.
\end{align}
where $C_1(N,s)$ is a positive constant depending only on $N$ and $s$.
\end{proposition}

\begin{proof}
First of all, direct calculations and Proposition \ref{prop:analy} give
 \begin{equation}
 \label{eq:wt}
 w_t=A \omega, \qquad \omega_{tt}= A^2 \omega, \qquad L^s_m \omega = ( -|\lambda|^2 +m^ 2 )^ s\omega,\qquad  L^s_m (\omega_t)= ( -|\lambda|^2 +m^ 2 )^ s A \omega,
 \end{equation}
 \begin{equation}
 \label{eq:Lmsw}
L_m^{2s}w= ( -|\lambda|^2 +m^ 2 )^{2s}\omega,
\end{equation}
and
$$H^s_m(\omega,\omega) =\big( (-4 |\lambda|^2+m^2 )^s - 2 (- |\lambda|^2+m^2 )^s  \big) \omega^2.$$
Let $\mathcal{H}(t)$ be defined as in \eqref{eq:Ht}.
Then
$$
\dot{\mathcal{H}}(t) = \int_{\RN} \omega_t u^2 \,dx + 2\int_{\RN} \omega u u_t  \,dx=\int_{\RN} (\omega_t - L^s_m \omega) u^2  \,dx   + \int_{\RN} \omega H^s_m(u,u)  \,dx +  2 \int_{\RN} \omega u F  \,dx.
$$
Thus
\begin{equation}
\label{eq:Hpunt}
\dot{\mathcal{H}}(t) = D(t) + 2 \int_{\RN} \omega u F dx.
\end{equation}

We focus on $D(t)$. We have, by using the relation \eqref{eq:CC} and after tedious computations,
\begin{align}\label{Ddot}
\dot{D}(t)&= \int_{\RN} (\omega_{tt} -2 L^s_m(\omega_t) +L^{2s}_m \omega) u^ 2  \,dx + 2 \int_{\RN}\omega_t H^s_m(u,u) \,dx +2\int_{\RN} ( \omega_t  -  L_m^s\omega) \cdot u F  \,dx
\\
\notag & \quad + 2  \int_{\RN}  \omega ( L^s_m u -F)^2 \,dx - 2\int_{\RN} \omega F^2 \,dx     - \int_{\RN} \omega H^{2s}_m(u,u)   \,dx- 2 \int_{\RN}  H^s_m (\omega, u)  F \,dx.
\end{align}
In view of the expression of $\omega$ in \eqref{eq:Ht}, we have that, by \eqref{eq:wt} and \eqref{eq:Lmsw},
\begin{equation}\label{Ddot:H}
 \int_{\RN} (\omega_{tt} -2 L^s_m(\omega_t) +L^{2s}_m \omega ) u^ 2 \,dx = \big((-|\lambda|^2+m^2)^s-A\big)^2\mathcal{H}(t)
\end{equation}
and
\begin{equation}\label{Ddot:Hs}
 2 \int_{\RN}\omega_t H^s_m(u,u)\,dx=2A\int_{\RN} \omega  H^s_m(u,u)\,dx.
\end{equation}
Concerning the term $2\int_{\RN} ( \omega_t  -  L_m^s\omega) \cdot u F  \,dx  $, by \eqref{eq:wt} and by applying arithmetic-geometric inequality (AM-GM inequality), we get
\begin{align*}
\notag 2\Big|\int_{\RN} ( \omega_t  -  L_m^s\omega) u F  \,dx \Big|&=2\Big|\int_{\RN} \big(A-(-|\lambda|^2+m^2)^s\big)\omega u F  \,dx \Big|\\
&\le \frac14\big((-|\lambda|^2+m^2)^s-A\big)^2\mathcal{H}(t)+4\int_{\RN}\omega F^2\,dx.
\end{align*}

Let us see now how to estimate the term $ -2 \int  H^s_m (\omega, u)  F$ in \eqref{Ddot}.
By definition we have that
\begin{align*}
& -2 \int_{\mathbb{R}^N}   H^s_m (\omega, u)  F\, dx \\
&\quad = 2C_{N,s} m^{\frac{N+2s}{2}}e^{At}\int_{\mathbb{R}^N} \int_{\mathbb{R}^N} \frac{\big(e^{\lambda\cdot x}-e^{\lambda\cdot y}\big)\big(u(t,x)-u(t,y)\big)}{|x-y|^{\frac{N+2s}{2}}}
K_{\frac{N+2s}{2}}(m|x-y|)\,dy F\,dx+2m^{2s}\int_{\RN}Fu\omega\,dx \\
& \quad = 2e^{At}(I_{|x-y|<1/m} +  I_{|x-y|>1/m})+2m^{2s}\int_{\RN}Fu\omega\,dx,
\end{align*}
where the integrals $I_{|x-y|<1/m}$ and $I_{|x-y|>1/m}$ are determined by the splitting $ \int_{\mathbb{R}^N} \int_{|x-y|<1/m} $ and $ \int_{\mathbb{R}^N} \int_{|x-y|>1/m} $, respectively.
The integral close to the origin is bounded as follows (we use the asymptotics of Macdonald's function in \eqref{eq:asymp0} that involve constants depending on $s$ that do not blow up)
\begin{align*}
I_{|x-y|<1/m}&\simeq c\, C_{N,s}  \Gamma\Big(\frac{N+2s}{2} \Big)  2^{\frac{N+2s}{2}-1}  m^{\frac{N+2s}{2}} \\
&\quad \times \int_{\mathbb{R}^N} \int_{|x-y|<1/m}  e^{\lambda \cdot x}  \frac{\big(1-e^{\lambda \cdot(y-x)}\big)\big(u(t,x)-u(t,y)\big)}{|x-y|^{\frac{N+2s}{2}}}
(m|x-y|)^{-\frac{N+2s}{2}} \,dy F \,dx\\
&=c\,C_{N,s}  \Gamma\Big(\frac{N+2s}{2} \Big)  2^{\frac{N+2s}{2}-1} m^{\frac{N+2s}{4}} \\
&\quad \times\int_{\mathbb{R}^N}\int_{|x-y|<1/m} e^{\lambda \cdot x} \frac{1-e^{\lambda \cdot(y-x)}}{|x-y|^{\frac{N+2s}{2}}} \frac{\big(u(t,x)-u(t,y)\big)}{|x-y|^{\frac{N+2s}{4}}}
(m|x-y|)^{-\frac{N+2s}{4}} \,dy \, F\, dx.
\end{align*}
Using again the asymptotics for the Macdonald's function, and applying Cauchy-Schwartz in the integral in the variable $y$, we get
\begin{align*}
|I_{|x-y|<1/m}|& \le   c \bigg(\int_{|z|<1/m} \frac{ (1-e^{\lambda \cdot z})^2 }{|z|^{N+2s}} dz\bigg)^{1/2}  \cdot C_{N,s}  \Gamma\Big(\frac{N+2s}{2} \Big)  2^{\frac{N+2s}{2}-1}
m^{\frac{N+2s}{4}} \int_{\mathbb{R}^N} e^{\lambda\cdot x} |F|  \\
&\quad \times \bigg(\int_{|x-y|<1/m}  \frac{\big(u(t,x)-u(t,y)\big)^2}{|x-y|^{\frac{N+2s}{2}}}  K_{\frac{N+2s}{2}}(m|x-y|)
\frac{1}{\Gamma(\frac{N+2s}{2})} 2^{-\frac{N+2s}{2}+1}     \,dy \bigg)^{1/2} \,dx  \\
&\le  c\bigg(\int_{|z|<1/m} \frac{ (1-e^{\lambda \cdot z})^2 }{|z|^{N+2s}} dz\bigg)^{1/2}  \cdot C_{N,s}  \Gamma\Big(\frac{N+2s}{2} \Big)  2^{\frac{N+2s}{2}-1} m^{\frac{N+2s}{4}}
 \\
&\quad \times \frac{1}{(\Gamma(\frac{N+2s}{2}))^{1/2}} 2^{-\frac{N+2s}{4}+1/2}\int_{\mathbb{R}^N} e^{\lambda \cdot x} |F | \Big(-H^{s}_m(u,u)  m^{-\frac{N+2s}{2}}\frac{1}{C_{N,s}}  \Big)^{1/2} \,dx  \\
&=c\bigg(\int_{|z|<1/m} \frac{ (1-e^{\lambda \cdot z})^2 }{|z|^{N+2s}} \,dz\bigg)^{1/2} C_{N,s}  \Gamma\Big(\frac{N+2s}{2} \Big)  2^{\frac{N+2s}{2}-1}
 \\
&\quad \times \frac{1}{(\Gamma(\frac{N+2s}{2}))^{1/2}} 2^{-\frac{N+2s}{4}+1/2}\frac{1}{(C_{N,s})^{1/2}}  \int_{\mathbb{R}^N} e^{\lambda\cdot x} |F |(-H^{s}_m(u,u) )^{1/2} \,dx,
\end{align*}
where we used Proposition \ref{prop:leibniz}.
Let us estimate the quantity
\begin{equation}
\label{quant}
c\frac12\bigg[ \bigg(\int_{|z|<1/m} \frac{ (1-e^{\lambda \cdot z})^2 }{|z|^{N+2s}}\, dz\bigg)^{1/2}  C_{N,s}  \Gamma\Big(\frac{N+2s}{2} \Big)  2^{\frac{N+2s}{2}-1}
\frac{1}{(\Gamma(\frac{N+2s}{2}))^{1/2}} 2^{-\frac{N+2s}{4}+1/2}\frac{1}{(C_{N,s})^{1/2}} \bigg] ^2.
\end{equation}
We will prove that for $|\lambda|<m$, the positive constant \eqref{quant} is bounded from above independently of $\lambda$ and $m$. Indeed, for $|\lambda|<m$, using the mean value
theorem we obtain
$$
 \int_{|z|<1/m} \frac{ (1-e^{\lambda \cdot z})^2 }{|z|^{N+2s}}\,dy \le
\int_{|z|<1/m} \frac{ e^{2\frac{|\lambda|}{m}} |\lambda|^2 |z|^2 }{|z|^{N+2s}} \,dz = e^{2\frac{|\lambda|}{m}} |\lambda|^2 \frac{1}{2 m^{2-2s}}\omega_N \le e^2 \frac{1}{2}\omega_Nm^{2s}.
$$
We take into account that $C_{N,s}$ is given by the formula \eqref{constant}, thus the constant \eqref{quant} above is bounded by
$$
 c\,e^2 \frac{1}{4}\omega_N C_{N,s}  \Gamma\Big(\frac{N+2s}{2} \Big)  2^{\frac{N+2s-2}{2}} = -c\frac{1}{2} e^2 \omega_N\frac{4^s}{\pi^{N/2}}\frac{ \Gamma\big(\frac{N+2s}{2} \big)}{\Gamma(-s)}m^{2s}=:C_1(N,s)m^{2s}.
$$
Now we apply AM-GM in the variable $x$ to obtain
\begin{equation}\label{Iorigen:bound}
|I_{|x-y|<1/m}|\le -\frac{m^{2s}}{2}\int_{\RN}e^{\lambda\cdot x} H^{s}_m(u,u) \,dx  +   C_1(N,s)\int_{\mathbb{R}^N} e^{\lambda\cdot x} F^2\,dx.
\end{equation}

For the integral away from the origin we have
\begin{align*}
I_{|x-y|
>1/m}&=C_{N,s}    m^{\frac{N+2s}{2}} \\
&\quad \times\int_{\mathbb{R}^N} \int_{|x-y|>1/m}  \frac{e^{\lambda \cdot x} u(t,x) -e^{\lambda\cdot x} u(t,y) -e^{\lambda\cdot y}u(t,x) +e^{\lambda\cdot y} u(t,y)}{|x-y|^{\frac{N+2s}{2}}}
K_{\frac{N+2s}{2}}(m|x-y|) \,dy F\, dx\\
& = E_ 1+ E_2 .
\end{align*}
Concerning $E_1$, we get
\begin{align*}
E_1 &= C_{N,s}    m^{\frac{N+2s}{2}} \int_{\mathbb{R}^N} \int_{|x-y|>1/m}  \frac{ ( e^{\lambda\cdot x}-e^{\lambda\cdot y})  u(t,x)  }{|x-y|^{\frac{N+2s}{2}}}
K_{\frac{N+2s}{2}}(m|x-y|) \,dy  F \,dx \\
&= C_{N,s}    m^{\frac{N+2s}{2}}  \int_{|z|>1/m}  \frac{1-e^{\lambda \cdot z}}{|z|^{\frac{N+2s}{2}}}  K_{\frac{N+2s}{2}}(m|z|) \,dz \cdot  \int_{\mathbb{R}^N} e^{\lambda\cdot x}
u(t,x) F\, dx.
\end{align*}
Observe that
$$
\int_{|z|>1/m}  \frac{1-e^{\lambda \cdot z}}{|z|^{\frac{N+2s}{2}}} K_{\frac{N+2s}{2}}(m|z|) \,dz = m^{\frac{N+2s}{2}-N}\int_{|y|>1}  \frac{1-e^{\frac{\lambda}{m} \cdot
z}}{|y|^{\frac{N+2s}{2}}} K_{\frac{N+2s}{2}}(|y|) \,dy.
$$
This integral is finite as proved in Lemma \ref{lemma:Bessel}. Indeed,
\begin{align*}
\bigg|\int_{|z|>1/m}  \frac{1-e^{\lambda \cdot z}}{|z|^{\frac{N+2s}{2}}} K_{\frac{N+2s}{2}}(m|z|) \,dz\bigg| \le \bigg|\int_{\RN}  \frac{1-e^{\lambda \cdot z}}{|z|^{\frac{N+2s}{2}}} K_{\frac{N+2s}{2}}(m|z|) \,dz\bigg|&=
m^{\frac{-N+2s}{2}} \Big|\Big (1-\Big(\frac{|\lambda|}{m}\Big)^2\Big)^{s
}-1\Big|\\
&\le 2 m^{\frac{-N+2s}{2}}.
\end{align*}
Thus, applying AM-GM inequality, we get
\begin{equation}\label{E1:bound}
|E_1|\le C_{N,s} m^{2s}2 \int_{\mathbb{R}^N} e^{\lambda\cdot x}| u|   \, |F|\, dx \le C_{N,s}^2 2  m^{4s}\, \int_{\mathbb{R}^N} e^{\lambda\cdot x} u^2\,dx  +\frac{1}{2}\int_{\RN}
e^{\lambda\cdot x}  F^2\, dx.
\end{equation}

For $E_2$ we have
\begin{align*}
|E_2| &\le C_{N,s}    m^{\frac{N+2s}{2}} \int_{\mathbb{R}^N} \int_{|x-y|>1/m}  \frac{ (e^{\lambda\cdot x} -e^{\lambda\cdot y}) |u(t,y)  |}{|x-y|^{\frac{N+2s}{2}}}
K_{\frac{N+2s}{2}}(m|x-y|) \,dy  |F |\,dx \\
&= C_{N,s}    m^{\frac{N+2s}{2}}\int_{\mathbb{R}^N}  \int_{|x-y|>1/m}  e^{\frac{1}{2}\lambda \cdot y} |u(t,y)  |   \frac{  e^{ -\frac{1}{2}\lambda\cdot (x+y)  }  (e^{\lambda\cdot x}-
e^{\lambda \cdot y)} }{|x-y|^{\frac{N+2s}{2}}}  K_{\frac{N+2s}{2}}(m|x-y|) \,dy \, \cdot  e^{\frac{1}{2}\lambda \cdot x} |F |\,dx \\
&\le  C_{N,s}    m^{\frac{N+2s}{2}}\bigg( \int_{\mathbb{R}^N}  e^{\lambda \cdot x} F^2\, dx \bigg)^{1/2}  \\
&\quad  \times \bigg( \int_{\RN}  \bigg[
\int_{|x-y|>1/m}  e^{\frac{1}{2}\lambda \cdot y} |u(t,y)  |    \frac{   e^{\frac{1}{2}\lambda (x-y)}- e^{\frac{1}{2}\lambda\cdot (y-x)}}{|x-y|^{\frac{N+2s}{2}}}
K_{\frac{N+2s}{2}}(m|x-y|) \,dy \bigg]^2 \,dx \bigg)^{1/2} \\
&\le C_{N,s}    m^{\frac{N+2s}{2}} \bigg( \int_{\mathbb{R}^N}  e^{\lambda \cdot x} F^2 \,dx \bigg)^{1/2}  \cdot \bigg( \int_{\RN}  \bigg[
\int_{\RN}  e^{\frac{1}{2}\lambda \cdot y}|u(t,y)  |\cdot  k(x-y)\,dy \bigg]^2\, dx \bigg)^{1/2} \\
&= C_{N,s}    m^{\frac{N+2s}{2}} \bigg( \int_{\mathbb{R}^N}  e^{\lambda \cdot x} F^2\, dx \bigg)^{1/2}  \cdot \| e^{\frac{1}{2}\lambda \cdot  (\cdot)} u\ast  k\|_{L^2(\RN)}  \\
&\le C_{N,s}    m^{\frac{N+2s}{2}} \bigg( \int_{\mathbb{R}^N}  e^{\lambda \cdot x} F^2\, dx \bigg)^{1/2}  \cdot \bigg( \int_{\mathbb{R}^N}  e^{\lambda \cdot x} u^2\,dx
\bigg)^{1/2}  \cdot \|  k(x)\|_{L^1(\RN)},
\end{align*}
where $k(z)=\chi_{\{|z|>{1/m}\}} (|z|) \cdot \frac{   e^{ \frac{1}{2}\lambda\cdot z}- e^{- \frac{1}{2}\lambda\cdot z}}{|z|^{\frac{N+2s}{2}}}  K_{\frac{N+2s}{2}}(m|z|).$
Observe that
if $|\lambda|<m$ then
$$
\|  k(x)\|_{L^1(\RN)} \le  2 \sqrt{\frac{\pi}{2}}           m^{-1/2}\cdot  \int_{\{|x|>1/m\}} \frac{  1  } {|x|^{\frac{N+2s+1}{2}}}   e^{-\frac{m}{2}|x|} \,dx=\sqrt{\pi}   2^{\frac{N-2s}{2}}  \  m^{\frac{-N+2s}{2}}   \Gamma \Big(\frac{N-2s-1}{2}\Big).
$$
We use AM-GM to get
\begin{equation}\label{E2:bound}
|E_2| \le C_{N,s} \Big(\sqrt{\pi} 2^{\frac{N-2s}{2}}    \Gamma \Big(\frac{N-2s-1}{2}\Big) \Big)^2 m^{4s}  \int_{\mathbb{R}^N}  e^{\lambda \cdot x} u^2\,dx     + \frac{1}{2}
\int_{\mathbb{R}^N}  e^{\lambda \cdot x} F^2 \,dx.
\end{equation}
Thus using \eqref{Iorigen:bound}, \eqref{E1:bound} and \eqref{E2:bound} we conclude that
\begin{align*}
\Big|-2 \int H^s_m (\omega, u)  F\,dx\Big| & \le -m^{2s} \int_{\RN}e^{At+\lambda \cdot x} H^{s}_m(u,u) \, dx  +  8 C_1(N,s)\int_{\mathbb{R}^N} e^{At+\lambda\cdot x} F^2\, dx \\
&\quad + C_2(N,s) m^{4s}  \, \int_{\mathbb{R}^N} e^{At+\lambda\cdot x} u^2\,dx  +  \int_{\mathbb{R}^N}  e^{At+\lambda \cdot x} F^2\, dx  \\
&\quad+  C_3(N,s) m^{4s} \int_{\mathbb{R}^N} e^{At+\lambda\cdot x} u^2\,dx  +\int_{\RN} e^{At+\lambda\cdot x}   F^2 \,dx\\
&\quad+4m^{4s} \int_{\mathbb{R}^N} e^{At+\lambda\cdot x} u^2\,dx  +\int_{\RN} e^{At+\lambda\cdot x}   F^2 \,dx,
\end{align*}
where $ C_1(N,s)$ is as above, and $ C_2(N,s)$ and $ C_3(N,s)$ are positive and depend only on $N$ and $s$.
We rename the constants (i.e., below $ C_1(N,s)$ and $ C_2(N,s)$ mean different constants, but still positive and depending only on $N$ and $s$) and we have
\begin{multline}\label{H:omega:v}
\Big|-2 \int H^s_m (\omega, u)  F\,dx\Big|\\ \le  - m^{2s} \int_{\RN}e^{At+\lambda \cdot x} H^{s}_m(u,u)\, dx  + C_1(N,s) \int_{\mathbb{R}^N} e^{At+\lambda \cdot x} F^2 \,dx
 +C_2(N,s) m^{4s}    \int_{\mathbb{R}^N}  e^{At+\lambda \cdot x}u^2\,dx.
\end{multline}

Summing up, from \eqref{Ddot}, \eqref{Ddot:H}, \eqref{Ddot:Hs} and \eqref{H:omega:v} we have
\begin{align*}
\dot{D}(t) &\ge
\frac34\big((-|\lambda|^2+m^2)^s-A\big)^2\mathcal{H}(t)  -C_1(N,s)\int_{\RN} \omega F^2  \,dx   + 2  \int_{\RN} \omega ( u_t)^2  \,dx \\
& \quad + (A+m^{2s}) \int_{\RN} H^s_m(u,u) \omega \, dx-\int_{\RN}\omega H^{2s}_m(u,u) \,dx   -C_2(N,s) m^{4s}    \int_{\mathbb{R}^N}  \omega u^2 \,dx.
\end{align*}
Thus if we take $A+m^{2s}<0$ sufficiently small (that is $|A|$ sufficiently large) such that $A$ satisfies \eqref{condA:2}, we get the conclusion.

Notice that since we want our energy terms to be positive, the presence of $H_m^{2s}$ imposes the restriction $s\in (0,1/2]$.
\end{proof}

\subsection{Carleman inequality}
\label{subsec:Carelx}

Once we have Proposition \ref{prop:lower} at our disposal, we will be able to prove Theorem~\ref{th11} in the Introduction. We state the result here again, with a slightly different reformulation.
\begin{theorem}
\label{carlemanCompleta}
Let $N\ge 1$, $s\in (0,1/2)$, $m>0$, and let $u_0 \in   \mathbb{H}^{2s}(\RN) \cap   L^2(e^{\lambda\cdot x}\, dx)$ for some $\lambda \in \RN$ with $|\lambda|\le m$. Assume that $F(t,x) \in L^2(0,T:L^2(e^{\lambda\cdot x}\, dx))$. Let $u$ be a  strong   solution to the initial value problem~\eqref{Eq:withF} and let  $\omega(t,x)=e^{At}e^{\lambda\cdot x}$.
Then, the following inequality holds
\begin{align} \label{Carleman1g}
\notag& \frac{1}{2}\int_0^1\mathcal{H}(t)\,dt + \frac{1}{2} \big((-|\lambda|^2+m^2)^s-A\big)^2 \int_0^1 t(1-t) \mathcal{H}(t) dt\\
 &\quad+ \frac{1}{2}\int_0^1 t(1-t)   \Big\{2  \int_{\RN} \omega ( u_t)^2dx  - \int \omega H^{2s}_m(u,u) dx  + (A+m^{2s})\int_{\RN} H^s_m(u,u) \omega \, dx\Big\} dt  \\
\notag&\le \frac{1}{2}\mathcal{H}(0)+\frac{1}{2}\mathcal{H}(1)
+  C_1(N,s)  \int_0^1 \int_{\RN} \omega\big( (\partial_t+L^s_m)(u) \big)^2\,dx\,dt,
\end{align}
for $A+m^{2s}<0$ sufficiently small (that is, $|A|$ sufficiently large) satisfying
$$
\big((-|\lambda|^2+m^2)^s-A\big)^2 \ge C_2(N,s) m^{4s},
$$
where $C_1(N,s),C_2(N,s)$ are positive constants depending only on $N$ and $s$.
\end{theorem}

\begin{proof}
We will consider the following tent function
$$
\eta(\tau)=
\begin{cases}
\frac{\tau}{t}, \quad &0\le \tau\le t,\\
\frac{1-\tau}{1-t}, \quad &t\le \tau\le 1.
\end{cases}
$$
Then $\dot \eta$ is a decreasing step function
$$
\dot{\eta}(\tau) + \frac{1}{1-t}=
\begin{cases}
\frac{1}{t(1-t)}, \quad &0\le \tau\le t,\\
0, \quad &t\le \tau\le 1.
\end{cases}
$$
Thus by denoting $\delta$ the distributional derivative of the Heaviside function with values $0$ and $1$  we obtain that
$\ddot{\eta} = - \frac{1}{t(1-t)}\delta_t$  in the distributional sense.
Let $\overline{\mathcal{H}}(t)=-(1-t)\mathcal{H}(0)-t\mathcal{H}(1)+\mathcal{H}(t)$, so that $\overline{\mathcal{H}}(0)=0$ and $\overline{\mathcal{H}}(1)=0$. Then, integrating by
parts we obtain
$$
\int_0^1\dot{\overline{\mathcal{H}}}(\tau)\dot{\eta}(\tau)\,d\tau=-\int_0^1\ddot{\eta}(\tau)\overline{\mathcal{H}}(\tau)\,d\tau
$$
and thus we infer that
$$
\mathcal{H}(t)=(1-t)\mathcal{H}(0)+t\mathcal{H}(1)+t(1-t)\int_0^1\dot{\mathcal{H}}(\tau)\dot{\eta}(\tau)\,d\tau.
$$
Then, taking into account \eqref{eq:Hpunt}, we have
\begin{equation*}
\mathcal{H}(t)=(1-t)\mathcal{H}(0)+t\mathcal{H}(1)
+t(1-t)\int_0^1\dot{\eta}(\tau)D(\tau)\,d\tau+2t(1-t)\int_0^1\dot{\eta}(\tau)\int_{\RN} \omega uF\,dx\,d\tau.
\end{equation*}
Integrating by parts,
\begin{equation*}
\mathcal{H}(t)=(1-t)\mathcal{H}(0)+t\mathcal{H}(1)
-t(1-t)\int_0^1\eta(\tau)\dot{D}(\tau)\,d\tau+2t(1-t)\int_0^1\dot{\eta}(\tau)\int_{\RN} \omega uF\,dx\,d\tau.
\end{equation*}
We integrate in $t$ between $0$ and $1$. Notice that
$$
\int_0 ^1 t(1-t) \eta(\tau) dt= \frac{1}{2} \tau (1-\tau)\quad \text{and} \quad
\int_0 ^1 t(1-t)\dot{\eta}(\tau) dt= \frac{1-2\tau}{2}.
$$
Then
\begin{equation*}
\int_0^1\mathcal{H}(t)dt=\frac{1}{2}\mathcal{H}(0)+\frac{1}{2}\mathcal{H}(1)
-\frac{1}{2}\int_0^1\tau (1-\tau)\dot{D}(\tau)\,d\tau+\int_0^1 (1-2\tau)\int_{\RN} \omega uF\,dx\,d\tau.
\end{equation*}
By renaming the integrals in $\tau$, this is equivalent to
\begin{equation*}
\int_0^1\mathcal{H}(t)dt+\frac{1}{2}\int_0^1t(1-t)\dot{D}(t)\,dt=\frac{1}{2}\mathcal{H}(0)+\frac{1}{2}\mathcal{H}(1)
+\int_0^1 (1-2t)\int_{\RN} \omega uF\,dx\,dt.
\end{equation*}
Now we use the estimate \eqref{Ddot:lowerbound} of Proposition \ref{prop:lower} (under the assumptions on $A$), so that
\begin{align*}
&\int_0^1\mathcal{H}(t)dt+\frac{1}{2}\int_0^1t(1-t)   \frac{3}{4} \big(A-(-|\lambda|^2+m^2)^s\big)^2\mathcal{H}(t) dt -C_1(N,s)\int_{\RN} t(1-t) \int_{\RN} \omega F^2\,dx \,dt\\
&\quad+ \frac{1}{2}\int_{\RN} t(1-t)   \Big\{2  \int_{\RN} \omega ( u_t)^2  -  \int \omega H^{2s}_m(u,u)   + (A+m^{2s})\int_{\RN} H^s_m(u,u) \omega \, dx\Big\}\,dt \\
&\le \frac{1}{2}\mathcal{H}(0)+\frac{1}{2}\mathcal{H}(1)
+\int_0^1 (1-2t)\int_{\RN} \omega uF\,dx\,dt,
\end{align*}
or, equivalently,
\begin{align*}
&\int_0^1\mathcal{H}(t)dt+\frac{1}{2}\int_0^1t(1-t)   \frac{3}{4} \big((-|\lambda|^2+m^2)^s-A\big)^2\mathcal{H}(t)\, dt \\
&\quad+ \frac{1}{2}\int_0^1 t(1-t)   \Big\{2  \int_{\RN} \omega ( u_t)^2  - \int_{\RN} \omega H^{2s}_m(u,u)   + (A+m^{2s})\int_{\RN} H^s_m(u,u) \omega \, dx\Big\} \\
\notag&\le \frac{1}{2}\mathcal{H}(0)+\frac{1}{2}\mathcal{H}(1)+C_1(N,s)\int_0^1 t(1-t) \int_{\RN} \omega F^2dx dt
+\int_0^1 (1-2t)\int_{\RN} \omega uF\,dx\,dt.
\end{align*}
Applying the AM-GM inequality  $(1-2t) \omega u F \le \frac{1}{2}(1-2t)^2  \omega F^2  + \omega u^2$ yields
\begin{align}\label{eq:pre}
\notag&\frac{1}{2}\int_0^1\mathcal{H}(t)\,dt+ \frac{3}{8} \big((-|\lambda|^2+m^2)^s-A\big)^2\int_0^1t(1-t) \mathcal{H}(t) \,dt \\
\notag&\qquad+ \frac{1}{2}\int_0^1 t(1-t)   \Big\{2  \int_{\RN} \omega ( u_t)^2  -\int_{\RN} \omega H^{2s}_m(u,u)   +  (A+m^{2s})\int_{\RN} H^s_m(u,u) \omega \, dx\Big\} \,dt  \\
&\quad \le \frac{1}{2}\mathcal{H}(0)+\frac{1}{2}\mathcal{H}(1)
+    \int_0^1  \Big(\frac{1}{2} (1-2t)^2  +C_1(N,s) t(1-t) \Big) \int_{\RN}\omega F^2\,dx\,dt.
\end{align}
Finally, by considering the maximum of the  weight functions in $t$ we obtain the conclusion.
\end{proof}
\begin{cor}Due to the monotonicity of $\mathcal{H}(t)$ as a function of $t$ (in particular, by \eqref{persistence:t:1}), we have that
$$\mathcal{H}(1) \le \int_0^1\mathcal{H}(t)dt + e^{A - ( -|\lambda|^2 +m^ 2 )^ s} \int_{0} ^1  \int_{\RN} \omega  F^2 \,dx\,dt.$$
 So the term $\frac{1}{2}\mathcal{H}(1)$ can be hidden in the left hand side into the term $\frac{1}{2}\int \mathcal{H}(t)\,dt$. Hence, the resulting terms turn out to be still
 positive, and we derive, from the Carleman inequality \eqref{Carleman1g}, that
\begin{align*}
& \frac{1}{2} \big((-|\lambda|^2+m^2)^s-A\big)^2 \int_0^1\mathcal{H}(t)t(1-t) dt  +   \text{positive energy terms}  \\
&\le \frac{1}{2}\mathcal{H}(0)
+  \big(C_1(N,s)+e^{A - ( -|\lambda|^2 +m^ 2 )^ s}\big)  \int_0^1 \int_{\RN} \omega\big( (\partial_t+L_m)(u) \big)^2\,dx\,dt,
\end{align*}
where $A$ satisfies  \eqref{condA:2}.
\end{cor}

\begin{remark}
Observe that, from \eqref{eq:pre}, we could also immediately deduce the following convexity inequality:
\begin{equation} \label{Carleman1Jems}
\|\sqrt{t(1-t)}\omega^{1/2}u\|_{L^2(\R^N\times[0,1])}\lesssim \sup_{t\in [0,1]}\|\omega^{1/2}F\|_{L^2(\R^N)}+
\mathcal{H}(0)+\mathcal{H}(1).
\end{equation}
The Carleman inequality in \eqref{Carleman1Jems} reminds the one contained in \cite[Lemma 4]{EKPV-JEMS}. Such an inequality is used therein to obtain a convexity inequality (see
\cite[Theorem 3]{EKPV-JEMS}).
\end{remark}


\appendix

\section{}
\label{equiv}

Apart from the definitions for $L_m^s$ given in Section \ref{Sect:Defn}, we introduce the definition using the subordination formula. Motivated by the formula \eqref{formula:spower},
we define the operator $ L_m^s (f)$ as follows. Let $0<s<1$,  $m\ge0$ and $f\in \mathcal{S}$. The operator $ L_m^s (f)$ is obtained as a weighted integral of the associated heat
semigroup, by means of the spectral theorem
\begin{equation}
\label{Defn:subordination}
L_m^{s}f(x) =\frac{1}{\Gamma(-s)} \int_0^\infty \big( e^{t (\Delta -m^2)} f(x) -f(x) \big) \frac{dt}{t^{1+s}}.
\end{equation}
We notice that the fractional power could be also defined using functional calculus as in Kato \cite[p. 286]{K}, Pazy \cite[p. 69]{P} or Yosida  \cite[p. 260]{Y}.

The following lemma is the analogous to \cite[Lemma 2.1]{Stinga} for the fractional relativistic operator.

\begin{lemma}
For $f\in \mathcal{S}(\RN)$ and $0<s<1$, the definitions given in \eqref{Defn:Fourier}, \eqref{Defn:subordination} and \eqref{eq:pointwise} are equivalent.
\end{lemma}
\begin{proof}
We will first prove that \eqref{Defn:Fourier} and \eqref{Defn:subordination} are equivalent.
Observe that, by the inverse Fourier transform,
$$
 e^{t (\Delta -m^2)} f(x) -f(x) =\frac{1}{(2\pi)^{N/2}}\int_{\RN}\big(e^{-t(|\xi|^2+m^2)}-1\big)\widehat{f}(\xi)e^{ix\cdot \xi}\,d\xi.
$$
With this and the change of variables $w=t(|\xi|^2+m^2)$ we obtain
\begin{align*}
\int_0^{\infty}\big| e^{t (\Delta -m^2)} f(x) -f(x)\big|\frac{dt}{t^{1+s}}&\le
C_N\int_0^{\infty}\int_{\RN}\big|e^{-t(|\xi|^2+m^2)}-1\big||\widehat{f}(\xi)|\,d\xi\frac{dt}{t^{1+s}}\\
&=C_N\int_{\RN}\int_0^{\infty}\big|e^{-w}-1\big|\frac{dw}{w^{1+s}}(|\xi|^2+m^2)^s|\widehat{f}(\xi)|\,d\xi\\
&=C_{s,N}\int_{\RN}(|\xi|^2+m^2)^s|\widehat{f}(\xi)|\,d\xi<\infty,
\end{align*}
since we are considering $f\in \mathcal{S}(\RN)$. Therefore, by Fubini's Theorem,
\begin{align*}
\frac{1}{\Gamma(-s)}\int_0^{\infty}\big( e^{t (\Delta -m^2)} f(x)
-f(x)\big)\frac{dt}{t^{1+s}}&\frac{1}{\Gamma(-s)}\frac{1}{(2\pi)^{N/2}}\int_{\RN}\int_0^{\infty}\big(e^{-t(|\xi|^2+m^2)}-1\big)\frac{dt}{t^{1+s}}\widehat{f}(\xi)e^{ix\cdot
\xi}\,d\xi\\
&=\frac{1}{\Gamma(-s)}\frac{1}{(2\pi)^{N/2}}\int_{\RN}\int_0^{\infty}\big(e^{-w}-1\big)\frac{dw}{w^{1+s}}(|\xi|^2+m^2)^s\widehat{f}(\xi)e^{ix\cdot \xi}\,d\xi\\
&=\frac{1}{(2\pi)^{N/2}}\int_{\RN}(|\xi|^2+m^2)^s|\widehat{f}(\xi)|\,d\xi=\mathcal{F}^{-1}\big((|\cdot|^2+m^2)\widehat{f}(\cdot)\big)(x).
\end{align*}

We will check the equivalence between \eqref{Defn:subordination} and \eqref{eq:pointwise}. Let us denote $W_{t,m}(x):=e^{-tm^2}\frac{e^{-\frac{|x-y|^2}{4t}}}{(4\pi t)^{N/2}}$. By
Fubini's Theorem,
\begin{align*}
\int_0^{\infty}\big( e^{t (\Delta -m^2)} f(x) -f(x)\big)\frac{dt}{t^{1+s}}&=\int_0^\infty \int_{\R^N}W_{t,m}(x-y)\big(f(y) -f(x)\big)\,dy \frac{dt}{t^{1+s}} \\
&\qquad +f(x)\int_0^\infty \Big( \int_{\R^N}  W_{t,m}(x-y)\,dy\Big)(1-e^{tm^2})\frac{dt}{t^{1+s}}.
\end{align*}
The integral in the second summand boils down to
$$
\int_0^\infty \Big( \int_{\R^N}  \frac{e^{-\frac{|x-y|^2}{4t}}}{(4\pi t)^{N/2}}\,dy\Big)(e^{-tm^2}-1)\frac{dt}{t^{1+s}}=\Gamma(-s)m^{2s},
$$
On the other hand, the integral in the first summand reads as
\begin{equation*}
\int_{\R^N} \big(f(y) -f(x)\big)\int_0^\infty e^{-tm^2}\frac{e^{-\frac{|x-y|^2}{4t}}}{(4\pi t)^{N/2}} \frac{dt}{t^{1+s}}\,dy
=\Gamma(-s)C_{N,s} m^{\frac{N+2s}{2}} \int_{\mathbb{R}^N} \frac{f(x)-f(y)}{|x-y|^{\frac{N+2s}{2}}} K_{\frac{N+2s}{2}}(m|x-y|)\,dy
\end{equation*}
where we used the integral representation \eqref{formula:Bessel} of the Macdonald's function $K_{\nu}$, after a change of variable. The applications of Fubini's theorem can be
justified following an analogous argument as in \cite[Lemma 2.1]{Stinga}, by using the asymptotics \eqref{eq:asymp0} and \eqref{eq:asympInf}.
\end{proof}

\section*{Acknowledgements}

The authors would like to express their gratitude to Aingeru Fern\'andez--Bertolin for a careful reading of the manuscript and to Pawe\l{}  Sztonyk for inspiring discussions on L\'evy processes.

The first and third authors were supported by the Basque Government through BERC 2018--2021 program and by Spanish Ministry of Science, Innovation and Universities through BCAM Severo Ochoa accreditation SEV-2017-2018. The second and third authors were supported by the ERCEA Advanced Grant 2014 669689 - HADE and by the Spanish research project PGC2018-094522 N-100 from the MICINNU. The first author is also supported by the project PID2020-113156GB-I00 / AEI / 10.13039/501100011033 and acronym ``HAPDE''. She also acknowledges the RyC project RYC2018-025477-I and Ikerbasque.



\begin{thebibliography}{10}

\bibitem{BDQ}
D. Babusci, G. Dattoli, and M. Quattromini,
Relativistic equations with fractional and pseudodifferential operators.
\textit{Phys. Rev. A} \textbf{83} (2011), 062109.

\bibitem{Ba}
D. Bakry,
\'Etude des transformations de Riesz dans les vari\'et\'es riemanniennes \`a courbure de Ricci minor\'ee. (French) [A study of Riesz transforms in Riemannian manifolds with minorized Ricci curvature] \textit{S\'eminaire de Probabilit\'es, XXI}, 137--172, Lecture Notes in Math., \textbf{1247}, Springer, Berlin, 1987.

\bibitem{BGL}
D. Bakry, I. Gentil, and M. Ledoux,
\textit{Analysis and geometry of Markov diffusion operators.}
Grundlehren der Mathematischen Wissenschaften [Fundamental Principles of Mathematical Sciences], \textbf{348}. Springer, Cham, 2014.

\bibitem{BG}
A. Banerjee and N. Garofalo,
Monotonicity of generalized frequencies and the strong unique continuation property for fractional parabolic equations.
\textit{Adv. Math.} \textbf{336} (2018), 149--241.




\bibitem{CMS}
 R. Carmona, W. C. Masters, and B. Simon,
 Relativistic Schr\"odinger operators: asymptotic behavior of the eigenfunctions.
 \textit{J. Funct. Anal.} \textbf{91} (1990), no. 1, 117--142.

 \bibitem{CKK}
Z.-Q. Chen, P. Kim, and T. Kumagai,
Global heat kernel estimates for symmetric jump processes.
\textit{Trans. Amer. Math. Soc.} \textbf{363} (2011), no. 9, 5021--5055.

 \bibitem{CKS}
Z.-Q. Chen, P. Kim, and R. Song,
Sharp heat kernel estimates for relativistic stable processes in open sets.
\textit{Ann. Probab.} \textbf{40} (2012), 213--244.



\bibitem{EKPV07}
 L. Escauriaza, C. E. Kenig, G. Ponce, and L. Vega,
 On uniqueness properties of solutions of the $k$-generalized KdV equations, \textit{J. Funct. Anal. } \textbf{244} (2007) 504--535.

  \bibitem{EKPV-JEMS}
 L. Escauriaza, C. E. Kenig, G. Ponce, and L. Vega,
 Hardy's uncertainty principle, convexity and Schr\"odinger evolutions. \textit{J. Eur. Math. Soc. (JEMS)} \textbf{10} (2008), no. 4, 883--907.

\bibitem{EKPV-CMP}
 L. Escauriaza, C. E. Kenig, G. Ponce, and L. Vega,
 Hardy uncertainty principle, convexity and parabolic evolutions. \textit{Comm. Math. Phys.} \textbf{346} (2016), no. 2, 667--678.
%
%


\bibitem{FF1}
 M. M. Fall and V. Felli,
 Unique continuation property and local asymptotics of solutions to fractional elliptic equations.
 \textit{Comm. Partial Differential Equations} \textbf{39} (2014), no. 2, 354--397.

\bibitem{FF2}
 M. M. Fall and V. Felli,
 Unique continuation properties for relativistic Schr\"odinger operators with a singular potential.
 \textit{Discrete Contin. Dyn. Syst.} \textbf{35} (2015), no. 12, 5827--5867.


\bibitem{F-BV}
 A. Fern\'andez--Bertolin and L. Vega,
Uniqueness properties for discrete equations and Carleman estimates.
\textit{J. Funct. Anal.} \textbf{ 272} (2017), no. 11, 4853--4869.


  \bibitem{FLS}
 R. L. Frank, E. H. Lieb, and R. Seiringer,
 Hardy--Lieb--Thirring inequalities for fractional Schr\"odinger operators. \textit{J. Amer. Math. Soc.} \textbf{21} (2008), no. 4, 925--950.

\bibitem{FLSi} R. L. Frank, E. Lenzmann, and L. Silvestre, Uniqueness of radial solutions for the fractional Laplacian. \textit{Comm. Pure Appl. Math.}, \textbf{69} (2016), no. 9, 1671--1726.


 \bibitem{FJL}
J. Fr\"ohlich, B. L. G. Jonsson, and E. Lenzmann,
Boson stars as solitary waves.
\textit{Comm. Math. Phys.} \textbf{274} (2007), no. 1, 1--30.


\bibitem{GL}
N. Garofalo and F.-H. Lin, Monotonicity properties of variational integrals, $A_p$ weights and unique continuation, \textit{Indiana Univ. Math. J.}, \textbf{35} (1986), no. 2, 245--268.

%

\bibitem{GrRy}
T. Grzywny and M. Ryznar,
Two-sided optimal bounds for Green functions of half-spaces for relativistic $\alpha$-stable process.
\textit{Potential Anal.} \textbf{28} (2008), no. 3, 201--239.


\bibitem{K}
T. Kato,
\textit{Perturbation theory for linear operators. Reprint of the 1980 edition. Classics in Mathematics.}
Springer-Verlag, Berlin, 1995.


\bibitem{KW}  P.-Z. Ko and J.-N. Wang, 
Landis-type conjecture for the half-Laplacian, 	arXiv:2106.06120. 


\bibitem{KK} V. Knopova and A. M. Kulik,
Exact asymptotic for distribution densities of L\'evy functionals.
\textit{Electron. J. Probab.}
\textbf{16} (2011), no. 52, 1394--1433.


\bibitem{L}
C. L\"ammerzahl,
The pseudodifferential operator square root of the Klein--Gordon equation.
\textit{J. Math. Phys.} \textbf{34} (1993), no. 9, 3918--3932.



\bibitem{Lebedev}
N. N. Lebedev,
\textit{Special Functions and Its Applications},
Dover, New York, 1972.

%


\bibitem{LL}
E. H. Lieb and M. Loss,
\textit{Analysis}. Second edition. Graduate Studies in Mathematics, \textbf{14}. American Mathematical Society, Providence, RI, 2001.

\bibitem{M}
V. Z. Meshkov,
On the possible rate of decrease at infinity of the solutions of second-order partial differential equations. (Russian)
\textit{Mat. Sb.} \textbf{182} (1991), no. 3, 364--383; translation in \textit{Math. USSR-Sb.} \textbf{72} (1992), no. 2, 343--361.

\bibitem{NUV}
A.F. Nikiforov and V.B. Uvarov,
\textit{Special functions of mathematical physics. A unified introduction with applications.} Translated from the Russian and with a preface by Ralph P. Boas. With a foreword by A.
A. Samarski\u{\i}. Birkh\"auser Verlag, Basel, 1988.

\bibitem{P}
A. Pazy,
\textit{Semigroups of linear operators and applications to partial differential equations.} Applied Mathematical Sciences, \textbf{44}. Springer-Verlag, New York, 1983.

\bibitem{R}
A. R\"uland,
Unique continuation for fractional Schr\"odinger equations with rough potentials.
\textit{Comm. Partial Differential Equations} \textbf{40} (2015), no. 1, 77--114.

\bibitem{RS}
A. R\"uland  and M. Salo, Quantitative approximation properties for the fractional heat equation, \textit{Math. Control Relat. Fields} \textbf{10} (2020), no. 1, 1--26.

\bibitem{RW}
A. R\"uland and J.-N. Wang,
On the fractional Landis conjecture,
\textit{J. Funct. Anal.} \textbf{277} (2019), no. 9, 3236--3270.

\bibitem{Ry}
M. Ryznar,
Estimates of Green function for relativistic $\alpha$-stable process. \textit{Potential Anal.} \textbf{17} (2002), no. 1, 1--23.

\bibitem{SKM}
S. G. Samko, A. A. Kilbas, and O. I. Marichev,
\textit{Fractional integrals and derivatives. Theory and applications. Edited and with a foreword by S. M. Nikol'ski\u{\i}. Translated from the 1987 Russian original. Revised by the authors.}
Gordon and Breach Science Publishers, Yverdon, 1993.

\bibitem{Spams}
I. Seo,
Unique continuation for fractional Schr\"odinger operators in three and higher dimensions. \textit{Proc. Amer. Math. Soc.} \textbf{143} (2015), no. 4, 1661--1664.

\bibitem{S}
I. Seo,
Carleman inequalities for fractional Laplacians and unique continuation. \textit{Taiwanese J. Math.} \textbf{19} (2015), no. 5, 1533--1540.

\bibitem{StSingular}
E. M. Stein,
\textit{Singular integrals and differentiability properties of functions.}
Princeton Mathematical Series, \textbf{30}. Princeton University Press, Princeton, NJ, 1970.


\bibitem{Stinga}
P. R. Stinga,
\textit{Fractional powers of second order partial differential operators: extension problem and regularity theory},
PhD Thesis, Universidad Aut\'onoma de Madrid, 2010.

\bibitem{Sz}
P. Sztonyk,
Transition density estimates for jump L\'evy processes.
\textit{Stochastic Process. Appl.} \textbf{121} (2011), 1245--1265.


\bibitem{Y}
K. Yosida,
\textit{Functional analysis.} Sixth edition. Grundlehren der Mathematischen Wissenschaften [Fundamental Principles of Mathematical Sciences], \textbf{123}. Springer-Verlag,
Berlin-New York, 1980.

\end{thebibliography}
\end{document}